\documentclass{amsart}
\usepackage{amsmath}
\usepackage{amsfonts, amssymb, euscript, mathrsfs}
\usepackage{amsthm, upref}
\usepackage{graphicx}
\usepackage[usenames, dvipsnames]{color}

\usepackage{tikz}
\usepackage{enumerate}
\usepackage[tmargin=1.1in,bmargin=1.0in,lmargin=0.8in,rmargin=1.2in]{geometry}
\usepackage{aliascnt}
\usepackage{hyperref}
\usepackage{verbatim}

\allowdisplaybreaks[4]

\newcommand{\BR}{\mathbb{R}}

\newcommand{\SL}{\sum\limits}

\newcommand{\al}{\alpha}

\newcommand{\ga}{\gamma}
\newcommand{\de}{\delta}

\newcommand{\ME}{\mathbf E}

\newcommand{\CF}{\mathcal F}

\newcommand{\CG}{\mathcal G}
\newcommand{\MP}{\mathbf P}

\newcommand{\CL}{\mathcal L}
\newcommand{\CS}{\mathcal S}

\newcommand{\CZ}{\mathcal Z}
\newcommand{\CX}{\mathcal X}
\newcommand{\CY}{\mathcal Y}
\newcommand{\CP}{\mathcal P}

\newcommand{\Oa}{\Omega}

\newcommand{\si}{\sigma}

\newcommand{\pa}{\partial}
\renewcommand{\phi}{\varphi}
\newcommand{\ta}{\theta}
\newcommand{\eps}{\varepsilon}

\newcommand{\Ra}{\Rightarrow}
\newcommand{\Lra}{\Leftrightarrow}

\newcommand{\ol}{\overline}
\newcommand{\CM}{\mathcal M}

\newcommand{\norm}[1]{\lVert#1\rVert}
\renewcommand{\comment}[1]{}

\newcommand{\btu}{\bigtriangleup}
\newcommand{\mP}{\mathbf{p}}
\newcommand{\dd}{\mathrm{d}}
\newcommand{\md}{\mathrm{d}}
\newcommand{\CV}{\mathcal V}

\DeclareMathOperator{\SRBM}{SRBM}

\DeclareMathOperator{\diag}{diag}

\DeclareMathOperator{\tr}{tr}

\DeclareMathOperator{\CBP}{CBP}

\begin{document}

\theoremstyle{plain}
\newtheorem{thm}{Theorem}[section]
\newtheorem*{thmnonumber}{Theorem}
\newtheorem{lemma}[thm]{Lemma}
\newtheorem{prop}[thm]{Proposition}
\newtheorem{cor}[thm]{Corollary}
\newtheorem{open}[thm]{Open Problem}

\theoremstyle{definition}
\newtheorem{defn}[thm]{Definition}
\newtheorem{asmp}[thm]{Assumption}
\newtheorem{notn}[thm]{Notation}
\newtheorem{prb}[thm]{Problem}

\theoremstyle{remark}
\newtheorem{rmk}[thm]{Remark}
\newtheorem{exm}[thm]{Example}
\newtheorem{clm}[thm]{Claim}

\large

\author{Andrey Sarantsev}

\title[Triple and Simultaneous Collisions of Competing Brownian Particles]{Triple and Simultaneous Collisions\\ of Competing Brownian Particles} 

\address{University of Washington, Department of Mathematics, Box 354350, Seattle, WA 98195-4350}

\email{ansa1989@math.washington.edu}

\date{January 29, 2015. Version 8}

\keywords{Reflected Brownian motion, competing Brownian particles, asymmetric collisions, named particles, ranked particles, triple collisions, simultaneous collisions, Skorohod problem, stochastic comparison, non-smooth parts of the boundary}

\subjclass[2010]{Primary 60K35, secondary 60J60, 60J65, 60H10, 91B26}

\begin{abstract}
Consider a finite system of competing Brownian particles on the real line. Each particle moves as a Brownian motion, with drift and diffusion coefficients depending only on its current rank relative to the other particles. A triple collision occurs if three particles are at the same position at the same moment. A simultaneous collision occurs if at a certain moment, there are two distinct pairs of particles such that in each pair, both particles occupy the same position. These two pairs of particles can overlap, so a triple collision is a particular case of a simultaneous collision. We find a necessary and sufficient condition for a.s. absense of triple and simultaneous collisions, continuing the work of Ichiba, Karatzas, Shkolnikov (2013). Our results are also valid for the case of asymmetric collisions, when the local time of collision between the particles is split unevenly between them; these systems were introduced in Karatzas, Pal, Shkolnikov (2012). 
\end{abstract}

\maketitle

\section{Introduction} 

This paper is devoted to finite systems of competing Brownian particles. First, let us informally describe these systems. Fix $N \ge 2$, the quantity of particles. Take real-valued parameters $g_1, \ldots, g_N$ and positive real-valued parameters $\si_1, \ldots, \si_N$. Consider $N$ particles, moving on the real line. At each time, rank them from the left to the right: the  particle which is currently the leftmost one has rank $1$, the second leftmost particle has rank $2$, etc., up to the rightmost particle, which has rank $N$. As particles move, they can exchange ranks. We shall explain below how to resolve ties between particles. The particles move according to the following law: for each $k = 1, \ldots, N$, the particle with (current) rank $k$ moves as a Brownian motion with drift coefficient $g_k$ and diffusion coefficient $\si_k^2$, for each $k = 1, \ldots, N$. Thus, the dynamics of each particle depends on its current rank among other particles.  

Let us now formally define these systems. Consider the standard setting: a filtered probability space $(\Oa, \CF, (\CF_t)_{t \ge 0}, \MP)$ with the filtration satisfying the usual conditions. 

Any one-dimensional Brownian motion with zero drift and unit diffusion coefficients starting from the origin is called a {\it standard Brownian motion}. The symbol $a'$ denotes the transpose of (a vector or a matrix) $a$. We write $1(C)$ for the indicator of the event $C$. 

For every vector $x  = (x_1, \ldots, x_N)' \in \BR^N$, let $\mP$ be the permutation on $\{1,\ldots, N\}$ with the following properties:

(i) it {\it orders the components of} $x$: $x_{\mP(i)} \le x_{\mP(j)}$ for $1 \le i \le j \le N$;

(ii) {\it ties are resolved in the lexicographic order}: if $1 \le i < j \le N$ and $x_{\mP(i)} = x_{\mP(j)}$, then $\mP(i) < \mP(j)$. 

There exists a unique permutation $\mP = \mP_x$ which satisfies these two properties. We shall call it {\it the ranking permutation} for the vector $x$.
For example, if $x = (1, -1, 0, 0)'$, then $\mP_x(1) = 2,\ \mP_x(2) = 3,\ \mP_x(3) = 4,\ \mP_x(4) = 1$. We write $x_{(i)} = x_{\mP_x(i)}$ for $i = 1, \ldots, N$, so that $x_{(1)} \le x_{(2)} \le \ldots \le x_{(N)}$ are the ranked components of the vector $x$. 

\begin{defn} Take i.i.d. standard $(\CF_t)_{t \ge 0}$-Brownian motions $W_1, \ldots, W_N$. For a continuous $\BR^N$-valued process 
$$
X = (X(t),\ t \ge 0),\ \ \ X(t) = (X_1(t), \ldots, X_N(t))',
$$
denote by $\mP_t \equiv \mP_{X(t)}$ the ranking permutation for the vector $X(t)$ for every $t \ge 0$. Suppose this process satisfies the following SDE:
\begin{equation}
\label{mainSDE}
dX_i(t) = \SL_{k=1}^N1(\mP_t(k) = i)\left[g_k\, \md t + \si_k\, \md W_i(t)\right],\ \ i = 1, \ldots, N.
\end{equation}
Then this process $X$ is called a {\it classical system of $N$ competing Brownian particles} with {\it drift coefficients} $g_1, \ldots, g_N$ and {\it diffusion coefficients} $\si_1^2, \ldots, \si_N^2$. For $k = 1, \ldots, N$, the process
$$
Y_k = (Y_k(t),\ t \ge 0),\ \ Y_k(t) := X_{\mP_t(k)}(t) \equiv X_{(k)}(t),
$$
is called the {\it $k$th ranked particle}. If $\mP_t(k) = i$, then we say that the particle $X_i(t) = Y_k(t)$ at time $t$ has {\it name} $i$ and {\it rank} $k$. 
\label{classical}
\end{defn}

These systems were introduced in the paper \cite{BFK2005} for financial modeling; on this topic, see subsection 1.5. The coefficients of the SDE~\eqref{mainSDE} are piecewise constant functions of $X_1(t), \ldots, X_N(t)$, so weak existence and uniqueness in law for such systems follow from \cite{Bass1987}. By definition, the ranked particles satisfy 
$$
Y_1(t) \le Y_2(t) \le \ldots \le Y_N(t).
$$

\begin{defn}
A {\it triple collision at time $t$} occurs if there exists a rank $k = 2, \ldots, N-1$ such that  $Y_{k-1}(t) = Y_{k}(t) = Y_{k+1}(t)$.
\label{d1}
\end{defn}

A triple collision is sometimes an undesirable phenomenon. For example,  existence and uniqueness of a strong solutions of the SDE~\eqref{mainSDE} has been proved only up to the first moment of a triple collision, see \cite[Theorem 2]{IKS2013}. In this paper, we give a necessary and sufficient condition for absence of triple collisions with probability one. First, let us define some related concepts. 

\begin{defn} A {\it simultaneous collision} at time $t$ occurs if there are ranks $k \ne l$ such that such that $Y_{k}(t) = Y_{k+1}(t),\ Y_{l}(t) = Y_{l+1}(t)$. 
\label{d2}
\end{defn}

A triple collision is a particular case of a simultaneous collision.

The main result of this article is as follows. 

\begin{thm} Consider a system from Definition~\ref{classical}. 

(i) Suppose the sequence $(\si_n^2)_{1 \le n \le N}$ is {\it concave}, that is,  
\begin{equation}
\label{maincond}
\si_{k+1}^2 - \si_k^2 \le \si_k^2 - \si_{k-1}^2,\ \ k = 2, \ldots, N-1.
\end{equation}
Then, with probability one, there are no triple and no simultaneous collisions at any time $t > 0$.

(ii) If the condition~\eqref{maincond} fails for a certain $k = 2, \ldots, N-1$, then with positive probability there exists a moment $t > 0$ such that there is a triple collision between particles with ranks $k-1$, $k$, and $k+1$ at time $t$. 
\label{classicalthm}
\end{thm}

The proof of this result is given in Section 4. We can state a remarkable corollary of this theorem. 

\begin{cor} Take a system from Definition~\ref{classical}. Suppose a.s. there are no triple collisions at any moment $t > 0$. Then a.s. there are no simultaneous collisions at any moment $t > 0$. 
\label{interesting}
\end{cor}

It is interesting that a system of $N = 4$ particles can have a.s. no simultaneous collisions of the form 
\begin{equation}
\label{1010}
Y_1(t) = Y_2(t),\ \ Y_3(t) = Y_4(t),
\end{equation}
and at the same time triple collisions with positive probability. For example, if you take 
$$
\si_1 = \si_4 = 1,\ \ \mbox{and}\ \ \si_2 = \si_3 = 1 - \eps\ \ \mbox{for sufficiently small}\ \ \eps > 0,
$$
then there are a.s. no simultaneous collisions of the form~\eqref{1010}, but with positive probability there is a triple collision of ranked particles $Y_1$, $Y_2$, and $Y_3$, and with positive probability there is a triple collision of ranked particles $Y_2$, $Y_3$, and $Y_4$.  If 
$$
\si_1 = \si_3 = 1,\ \ \mbox{and}\ \ \si_2 = \si_4 = 1 + \eps\ \ \mbox{for sufficiently small}\ \ \eps > 0,
$$
then there are a.s. no simultaneous collisions of the form~\eqref{1010}, and a.s. no triple collisions of ranked particles $Y_1$, $Y_2$, and $Y_3$, but with positive probability there is a triple collision of ranked particles $Y_2$, $Y_3$, and $Y_4$. This is shown in the companion paper \cite[Subsection 1.2]{MyOwn5}. 

\subsection{Collision local times}

Consider a system of competing Brownian particles from Definition~\ref{classical}. Define the processes $B_1 = (B_1(t), t \ge 0), \ldots, B_N = (B_N(t), t \ge 0)$ as follows:
$$
B_k(t) = \SL_{i=1}^N\int_0^t1(\mP_s(k) = i)\md W_i(s).
$$
One can calculate that $\langle B_i, B_j\rangle_t = \de_{ij}t$; therefore, these are i.i.d. standard Brownian motions. For $k = 2, \ldots, N$, let the process $L_{(k-1, k)} = (L_{(k-1, k)}(t),\ t \ge 0)$ be the semimartingale local time at zero of the nonnegative semimartingale $Y_k - Y_{k-1}$. For notational convenience, we let $L_{(0, 1)}(t) \equiv 0$ and $L_{(N, N+1)}(t) \equiv 0$. Then the ranked particles $Y_1, \ldots, Y_N$ satisfy the following dynamics:
\begin{equation}
\label{symmSDE}
Y_k(t) = Y_k(0) + g_kt + \si_kB_k(t) + \frac12L_{(k-1, k)}(t) - \frac12L_{(k, k+1)}(t),\ \ k = 1, \ldots, N.
\end{equation}
The equation~\eqref{symmSDE} was deduced in \cite[Lemma 1]{Ichiba11} and \cite[Theorem 2.5]{BG2008}; see also \cite[Section 3]{BFK2005} and \cite[Chapter 3]{IchibaThesis}. 

The process $L_{(k-1, k)}$ is called the {\it local time of collision between the particles $Y_{k-1}$ and $Y_k$}. One can regard the local time $L_{(k-1, k)}(t)$ to be the total amount of push between the $(k-1)$st and the $k$th ranked particles $Y_{k-1}$ and $Y_k$ accumulated by time $t$. This amount of push is necessary and sufficient to keep the particle $Y_k$ to the right of the particle $Y_{k-1}$, so that $Y_{k-1}(t) \le Y_k(t)$. 

When these two particles collide, the amount of push is split evenly between them: the amount $(1/2)L_{(k-1, k)}(t)$ goes to the right-sided particle $Y_k$ and pushes it to the right; the equal amount $(1/2)L_{(k-1, k)}(t)$ (with the minus sign) goes to the left-sided particle $Y_{k-1}$ and pushes it to the left. One possible physical interpretation of this phenomenon: the ranked particles  have the same mass; so, when they collide, they get the same amount of push. 

The local time process $L_{(k-1, k)}$ has the following properties: $L_{(k-1, k)}(0) = 0$, $L_{(k-1, k)}$ is nondecreasing, and it can increase only when $Y_{k-1}(t) = Y_k(t)$, that is, when particles with ranks $k-1$ and $k$ collide. We can formally write the last property as
\begin{equation}
\label{localtime}
\int_0^{\infty}1(Y_{k}(t) \ne Y_{k-1}(t))\md L_{(k-1, k)}(t) = 0.
\end{equation}

\subsection{Systems with asymmetric collisions}

If we change coefficients $1/2$ in~\eqref{symmSDE} to some other values, we get the model from the paper \cite{KPS2012}. The local times in this new model are split unevenly between the two colliding particles, as if they had different mass. Let us now formally define this model. First, let us describe its parameters. Let $N \ge 2$ be the quantity of particles. Fix real numbers $g_1, \ldots, g_N$ and positive real numbers $\si_1, \ldots, \si_N$, as before. In addition, fix real numbers $q^+_1$, $q^-_1, \ldots, q^+_N$, $q^-_N$, satisfying the following conditions:
$$
q^+_{k+1} + q^-_k = 1,\ \ k = 1, \ldots, N-1;\ \ 0 < q^{\pm}_k < 1,\ \ k = 1, \ldots, N.
$$

\begin{defn} Take i.i.d. standard $(\CF_t)_{t \ge 0}$-Brownian motions $B_1, \ldots, B_N$. Consider a continuous adapted $\BR^N$-valued process 
$$
Y = (Y(t),\ t \ge 0),\ \ \ Y(t) = (Y_1(t), \ldots, Y_N(t))',
$$
and $N-1$ continuous adapted real-valued processes
$$
L_{(k-1, k)} = (L_{(k-1, k)}(t),\ t \ge 0),\ \ k = 2, \ldots, N,
$$
with the following properties:

(i) $Y_1(t) \le \ldots \le Y_N(t),\ \ t \ge 0$, 

(ii) the process $Y$ satisfies the following system of equations:
\begin{equation}
\label{SDEasymm}
Y_k(t) = Y_k(0) + g_kt + \si_kB_k(t) + q^+_kL_{(k-1, k)}(t) - q^-_kL_{(k, k+1)}(t),\ \ \ k = 1, \ldots, N.
\end{equation}
We let $L_{(0, 1)}(t) \equiv 0$ and $L_{(N, N+1)}(t) \equiv 0$ for notational convenience. 

(iii) for each $k = 2, \ldots, N$, the process $L_{(k-1, k)} = (L_{(k-1, k)}(t),\ t \ge 0)$ has the properties mentioned above: $L_{(k-1, k)}(0) = 0$, $L_{(k-1, k)}$ is nondecreasing and satisfies~\eqref{localtime}.  

Then the process $Y$ is called {\it a system of $N$ competing Brownian particles with asymmetric collisions}, with {\it drift coefficients} $g_1, \ldots, g_N$, {\it diffusion coefficients} $\si_1^2, \ldots, \si_N^2$, and {\it parameters of collision} $q^{\pm}_1,\ldots, q^{\pm}_N$. For each $k = 1, \ldots, N$, the process $Y_k = (Y_k(t), t \ge 0)$ is called the {\it $k$th ranked particle}. For $k = 2, \ldots, N$, the process $L_{(k-1, k)}$ is called the {\it local time of collision between the particles $Y_{k-1}$ and $Y_k$}. 
\label{asymmdefn}
\end{defn}

The state space of the process $Y$ is $\mathcal W^N := \{y = (y_1, \ldots, y_N)' \in \BR^N\mid y_1 \le y_2 \le \ldots \le y_N\}$. Strong existence and pathwise uniqueness for $Y$ and $L$ are proved in \cite[Section 2.1]{KPS2012}.  

\begin{rmk} Triple and simultaneous collisions for these systems are defined similarly to Definitions~\ref{d1} and~\ref{d2}. 
\end{rmk}

In the case of asymmetric collisions, we can also define a corresponding named system of competing Brownian particles. 

\begin{defn} Consider a continuous adapted process 
$$
X = (X(t), t \ge 0),\ X(t) = (X_1(t), \ldots, X_N(t))'.
$$
Suppose $\mP_t$ is the ranking permutation of $X(t)$ for $t \ge 0$, as before, and
$$
Y_k(t) \equiv X_{\mP_k(t)}(t),\ \ k = 1,\ldots, N,\ t \ge 0,
$$
Let $L_{(k-1, k)} = (L_{(k-1, k)}(t), t \ge 0)$ be the semimartingale local time at zero of $Y_k - Y_{k-1}$, for $k = 2, \ldots, N$; and $L_{(0, 1)}(t) \equiv L_{(N, N+1)}(t) \equiv 0$ for notational convenience, as before. 

Then this system $X = (X_1, \ldots, X_N)'$ is governed by the following SDE: for $i = 1, \ldots, N$ and $t \ge 0$, 
\begin{align*}
\md X_i(t) & = \SL_{k=1}^N1(\mP_t(k) = i)\left(g_k\md t + \si_k\md W_i(t)\right) \\ & + \SL_{k=1}^N1(\mP_t(k) = i)\left(q^-_k - (1/2)\right)\md L_{(k, k+1)}(t)  \\ & 
+ \SL_{k=1}^N1(\mP_t(k) = i)\left(q^+_k - (1/2)\right)\md L_{(k-1, k)}(t).
\end{align*}
It is called a {\it system of named competing Brownian particles} with {\it drift coefficients} $(g_n)_{1 \le n \le N}$, {\it diffusion coefficients} $(\si_n^2)_{1 \le n \le N}$, and {\it parameters of collision} $(q^{\pm}_n)_{1 \le n \le N}$.
\label{namedasymm}
\end{defn}

The ranked particles $(Y_1, \ldots, Y_N)$ from Definition~\ref{namedasymm} form a system of ranked competing Brownian particles in the sense of Definition~\ref{asymmdefn}. 
However, unlike the system $Y$ from Definition~\ref{asymmdefn}, which exists and is unique in a strong sense up to the infinite time horizon, the system $X$ from Definition~\ref{namedasymm} is known to have strong solutions only up to the first moment of a triple collision, see \cite{KPS2012}. 
This provides a motivation to find a condition which guarantees absense of triple collisions. Here, we prove a necessary and sufficient condition for a.s. lack of triple collisions. 

\begin{thm}
Consider a system of competing Brownian particles with asymmetric collisions from Definition~\ref{asymmdefn}. 

(i) Suppose the following condition is true: 
\begin{equation}
\label{conditionasymm}
(q^-_{k-1} + q^+_{k+1})\si_k^2 \ge q^-_k\si^2_{k+1}  + q^+_k\si_{k-1}^2,\ \ k = 2,\ldots, N-1.
\end{equation}
Then, with probability one, there are no triple and no simultaneous collisions at any time $t > 0$.

(ii) If the condition~\eqref{conditionasymm} is violated for some $k = 2, \ldots, N-1$, then with positive probability there exists a moment $t > 0$ such that there is a triple collision between particles with ranks $k-1,\ k$, and $k+1$ at time $t$. 
\label{thmasymm}
\end{thm}

Note that Theorem~\ref{classicalthm} is a particular case of this theorem for $q^{\pm}_k = 1/2,\ k = 1, \ldots, N$. Corollary~\ref{interesting} is also true for systems with asymmetric collisions. 

\subsection{Method of proof: reduction to an SRBM in the orthant}

Let us informally describe a stochastic process called a {\it semimartingale reflected Brownian motion (SRBM) in the positive multidimensional orthant} $S := \BR^d_+$, where $\BR_+ := [0, \infty)$ and $d \ge 1$ is the dimension. We formally define an SRBM in subsection 2.1.

Fix the parameters of an SRBM: a {\it drift vector} $\mu \in \BR^d$, a {\it covariance matrix}: a $d\times d$-positive definite symmetric matrix $A = (a_{ij})_{1 \le i, j \le d}$, and a {\it reflection matrix}: a $d\times d$-matrix $R = (r_{ij})_{1 \le i, j \le d}$ with $r_{ii} = 1,\ i = 1, \ldots, d$. Then {\it an SRBM in the orthant $S$} with parameters $R, \mu, A$, denoted by $\SRBM^d(R, \mu, A)$, is a Markov process with state space $S$ which:

(i) behaves as a $d$-dimensional Brownian motion with drift vector $\mu$ and covariance matrix $A$ in the interior of the orthant $S$;

(ii) on each face $S_i = \{x \in S\mid x_i = 0\}$ of the boundary $\pa S$, the process is reflected in the direction of $r_i$, the $i$th column of $R$.

If $r_i = e_i$, where $e_i$ is the $i$th standard basis vector in $\BR^d$, then the reflection is called {\it normal}. Otherwise, it is called {\it oblique}. 

For a system of $N$ competing Brownian particles (the classical system or the one with asymmetric collisions), the gaps $Z_k(t) = Y_{k+1}(t) - Y_k(t),\ \ k = 1, \ldots, N-1$, between adjacent ranked particles form an SRBM in the orthant $\BR^{N-1}_+$: see subsection 2.2. If there is a simultaneous collision $Y_k(t) = Y_{k+1}(t)$ and $Y_l(t) = Y_{l+1}(t)$, then $Z_k(t) = Z_l(t) = 0$. In other words, a simultaneous collision is equivalent to the gap process hitting non-smooth parts of the boundary of the orthant $\BR^{N-1}_+$. In Theorem~\ref{SRBMnonsmooth}, subsection 2.3, we state a necessary and sufficient condition for an SRBM to a.s. avoid non-smooth parts of the boundary. This theorem is proved in Section 3. In Section 4, we translate these results into the language of competing Brownian particles, and prove Theorem~\ref{classicalthm} and Theorem~\ref{thmasymm}. 

We find whether this SRBM hits {\it non-smooth parts $S_i\cap S_j,\ i \ne j$ of the boundary} $\pa S$. This corresponds to triple or simultaneous collisions of competing Brownian particles. This connection is established in subsection 2.2.

\subsection{Relation to previous results}

For classical systems of competing Brownian particles from Definition~\ref{classical}, some significant partial results on the triple collision problem were known before. In particular, a necessary and sufficient condition for absence of triple collisions for systems with only three particles is obtained in \cite{IK2010}. In the article \cite{IKS2013}, it is proved that the condition~\eqref{maincond} from Theorem~\ref{classicalthm} is necessary. For systems with asymmetric collisions from \cite{KPS2012}, some sufficient conditions for absence of triple collisions were found, but these are not necessary conditions. 

In the companion paper \cite{MyOwn5}, we find a sufficient condition for avoiding collisions of four or more particles ({\it multiple collisions}), as well as {\it multicollisions}: when a few particles collide and at the same time other few particles collide. An example of a multicollision is $Y_1(t) = Y_2(t) = Y_3(t)$, $Y_5(t) = Y_6(t)$ and $Y_7(t) = Y_8(t)$. A simultaneous collision (for example, $Y_1(t) = Y_2(t)$ and $Y_3(t) = Y_4(t)$) is a particular case of a multicollision. In particular, as mentioned above, we can find examples of a system of four particles avoiding simultaneous collisions of the type~\eqref{1010} but having triple collisions with positive probability. 

We can also define a reflected Brownian motion in domains which are more general than the orthant. In particular, a {\it two-dimensional wedge} is a subset of $\BR^2$ of the form 
$$
\{(r\cos\ta, r\sin\ta)\mid 0 \le r < \infty,\ \ 0 \le \ta \le \xi\}
$$
(where $\xi \in (0, \pi)$ is the angle of this wedge). A reflected Brownian motion with unit drift vector and identity covariance matrix was studied in \cite{VW1985}. The latter paper provides a necessary and sufficient condition for this process to a.s. avoid hitting the origin (the corner of the wedge). In \cite{Wil1987b}, the Hausdorff dimension of the set of times when this process hits the corner was found. More generally, we can define a reflected Brownian motion in a convex polyhedron in $\BR^d$: see \cite{DK2003} and \cite{DW1995}. In \cite{Wil1987}, a reflected Brownian motion in a polyhedral domain was constructed under the so-called {\it skew-symmetry condition}, see~\eqref{SS} and~\eqref{SSgeneral}. It was shown that under this condition, it does not hit non-smooth parts of the boundary. These results are important and are applied in this article. 

Let us also mention some related sources on nonattainability of lower-dimensional manifolds by a diffusion process: the papers \cite{Friedman1974}, \cite{Ramasubramanian1983}, \cite{Ramasubramanian1988}, \cite{CepaLepingle}, and the book \cite{FriedmanBook}. 

\subsection{Motivation and historical review}

The original motivation to study classical systems of competing Brownian particles came from Stochastic Finance. An observed phenomenon of real-world stock markets is that stocks with smaller capitalizations have larger growth rates and larger volatilities. This can be captured by the classical model of competing Brownian particles: just let $g_1 > \ldots > g_N$ and $\si_1 > \ldots > \si_N$, and suppose that for $i = 1, \ldots, N$, the quantity $e^{X_i(t)}$ is the capitalization of the $i$th stock at time $t$. For financial applications and market models similar to this rank-based model, see the articles \cite{Ichiba11}, \cite{FIK2013b}, \cite{MyOwn4}, the book \cite[Chapter 5]{F2002} and a somewhat more recent survey \cite[Chapter 3]{FK2009}.

Classical systems from Definition~\ref{classical} were studied in \cite{IchibaThesis}, \cite{Ichiba11}, \cite{PP2008}, \cite{CP2010}, \cite{PS2010}, \cite{IPS2012}, \cite{FK2009}. There are several generalizations of these systems: 
\cite{S2011} (systems of competing Levy particles), \cite{PP2008}, \cite{IKS2013} (infinite systems of competing Brownian particles); \cite{FIK2013}, \cite{FIK2013b}, \cite{Ichiba11} ({\it second-order stock market models}, when drift and diffusion coefficients depend on both ranks and names).

Systems of competing Brownian particles with asymmetric collisions are related to the theory of exclusion processes: it was proved in \cite[Section 3]{KPS2012} that these systems are scaling limits of asymmetrically colliding random walks, which constitute a certain type of exclusion processes. In addition, thse systems are also related to random matrices and random surfaces evolving according to the KPZ equation, see \cite{FSW2013}.

Studying an SRBM in the orthant is motivated by queueing theory. An SRBM in the orthant is the {\it heavy traffic limit} for series of queues, when the traffic intensity at each queue tends to one, see \cite{ReimanThesis}, \cite{Rei1984}, \cite{Har1978}. We can also define an SRBM in general convex polyhedral domains in $\BR^d$, see \cite{DW1995}. An SRBM in the orthant and in convex polyhedra has been extensively studied, see the survey \cite{Wil1995}, and articles  \cite{HR1981a}, \cite{HR1981b}, \cite{HW1987a}, \cite{HW1987b}, \cite{Wil1987}, \cite{RW1988}, \cite{TW1993}, \cite{DW1994}, \cite{DK2003}, \cite{Chen1996}, \cite{Har1978}, \cite{BDH2010}, \cite{BL2007}, \cite{DW1995}, \cite{DH2012}, \cite{DH1992}, \cite{HH2009}, \cite{Wil1985a}, \cite{Wil1985b}, \cite{Wil1998b}, \cite{K1997}, \cite{K2000}, \cite{KW1996}, \cite{R2000}, \cite{KR2012a}, \cite{KR2012b}, \cite{KW1992a}, \cite{VW1985}, \cite{Wil1987b}.

\subsection{Notation} We denote by $I_k$ the $k\times k$-identity matrix. For a vector $x = (x_1, \ldots, x_d)' \in \BR^d$, let $\norm{x} := \left(x_1^2 + \ldots + x_d^2\right)^{1/2}$ be its Euclidean norm. 
For any two vectors $x, y \in \BR^d$, their dot product is denoted by $x\cdot y = x_1y_1 + \ldots + x_dy_d$. We compare vectors $x$ and $y$ componentwise: $x \le y$ if $x_i \le y_i$ for all $i = 1, \ldots, d$; $x < y$ if $x_i < y_i$ for all $i = 1, \ldots, d$; similarly for $x \ge y$ and $x > y$. We compare matrices of the same size componentwise, too. For example, we write $x \ge 0$ for $x \in \BR^d$ if $x_i \ge 0$ for $i = 1, \ldots, d$; $C = (c_{ij})_{1 \le i, j \le d} \ge 0$ if $c_{ij} \ge 0$ for all $i$, $j$. 

Fix $d \ge 1$, and let $I \subseteq \{1, \ldots, d\}$ be a nonempty subset. Write its elements in the order of increase: $I = \{i_1, \ldots, i_m\},\ \ 1 \le i_1 < i_2 < \ldots < i_m \le d$. For any $x \in \BR^d$, let
$[x]_I := (x_{i_1}, \ldots, x_{i_m})'$. For any $d\times d$-matrix $C = (c_{ij})_{1 \le i, j \le d}$, let $[C]_I := \left(c_{i_ki_l}\right)_{1 \le k, l \le m}$.

\section{Semimartingale Reflected Brownian Motion (SRBM) in the Orthant}

\subsection{Definition of an SRBM} Fix $d \ge 1$, the dimension. Recall that $\BR_+ := [0, \infty)$, and let $S := \BR^d_+$ be the $d$-dimensional positive orthant. Its boundary consists of $d$ {\it faces} $S_i = \{x \in S\mid x_i = 0\},\ i = 1, \ldots, d$. Take the parameters $R, \mu, A$ described in Subsection 1.2: a $d \times d$-matrix $R$ with diagonal elements equal to one, a $d\times d$ positive definite symmetric matrix $A$, and a vector $\mu \in \BR^d$. Assume the usual setting: a filtered probability space $(\Oa, \CF, (\CF_t)_{t \ge 0}, \MP)$ with the filtration satisfying the usual conditions. 

\begin{defn} Take a continuous function $\CX : \BR_+ \to \BR^d$ with $\CX(0) \in S$. A {\it solution to the Skorohod problem in the positive orthant $S$ with reflection matrix $R$ and driving function $\CX$} is a continuous function $\CZ : \BR_+ \to S$ such that there exists another continuous function $\CY : \BR_+ \to \BR^d$ with the following properties:

(i) for every $t \ge 0$, we have:
$\CZ(t) = \CX(t) + R\CY(t)$; 

(ii) for every $i = 1, \ldots, d$, the function $\CY_i$ is nondecreasing, satisfies $\CY_i(0) = 0$ and can increase only when $\CZ_i(t) = 0$, that is, when $\CZ(t) \in S_i$. We can write the last property formally as
$$
\int_0^{\infty}\CZ_i(t)\md \CY_i(t) = 0.
$$
\end{defn}

\begin{rmk} This definition can also be stated for a finite time horizon, that is, for functions $\CX, \CY, \CZ$ defined on $[0, T]$ instead of $\BR_+$. 
\end{rmk}

\begin{defn} Suppose $B = (B(t), t \ge 0)$ is an $((\CF_t)_{t \ge 0}, \MP)$-Brownian motion in $\BR^d$ with drift vector $\mu$ and covariance matrix $A$. A solution $Z = (Z(t), t \ge 0)$ to the Skorohod problem in $S$ with reflection matrix $R$ and driving function $B$ is called a {\it semimartingale reflected Brownian motion}, or SRBM, {\it in the positive orthant} $S$  with  {\it reflection matrix} $R$, {\it drift vector} $\mu$ and {\it covariance matrix $A$}. It is denoted by $\SRBM^d(R, \mu, A)$. The function $\CY$ is called the {\it vector of regulating processes}, and its $i$th component $\CY_i$ is called the {\it regulating process corresponding to the face $S_i$}. The process $B$ is called the {\it driving Brownian motion}. We say that $Z$ {\it starts from} $x \in S$ if $Z(0) = x$ a.s. 
\label{SRBMdefn}
\end{defn}

\begin{defn} Take a $d\times d$-matrix $R = (r_{ij})_{1 \le i, j \le d}$. It is called a {\it reflection matrix} if $r_{ii} = 1$ for $i = 1, \ldots, d$. It is called {\it nonnegative} if all its elements are nonnegative, that is, if $R \ge 0$; it is called {\it strictly nonnegative} if it is nonnegative and $r_{ii} > 0$ for $i = 1, \ldots, d$. It is called an {\it $\CS$-matrix} if there exists a vector $u \in \BR^d,\ u > 0$ such that $Ru > 0$. Any submatrix of $R$ of the form $[R]_I$, where $I \subseteq \{1, \ldots, d\}$ is a nonempty subset, is called a {\it principal submatrix} (this includes the matrix $R$ itself). The matrix $R$ is called {\it completely-$\CS$} if each of its principal submatrices is an $\CS$-matrix. It is called a $\CZ$-matrix if $r_{ij} \le 0$ for $i \ne j$. It is called {\it strictly inverse-nonnegative} if it is invertible and its inverse $R^{-1}$ is a strictly nonnegative matrix. It is called a {\it nonsingular $\CM$-matrix} if it is both completely-$\CS$ and a $\CZ$-matrix. 
\end{defn}

The following lemma is a useful characterization of reflection nonsingular $\CM$-matrices; its proof is given in the Appendix.

\begin{lemma} Suppose $R$ is a $d\times d$ reflection matrix. Then the following statements are equivalent:

(i) $R$ is a nonsingular $\CM$-matrix;

(ii) $R$ is a strictly inverse-nonnegative $\CZ$-matrix;

(iii) $R = I_d - Q$, where $Q$ is a nonnegative matrix with spectral radius less than $1$. 
\label{special}
\end{lemma}

We are ready to state an existence and uniqueness result, proved in \cite[Theorem 1]{HR1981a}, see also \cite[Theorem 2.1]{Wil1995}. This is not the most general result (for which the reader might want to see \cite{RW1988}, \cite{TW1993} and \cite[Theorem 2.3]{Wil1995}), but it is sufficient for our purposes. 

\begin{prop}
Suppose $R$ is a $d\times d$ reflection nonsingular $\CM$-matrix. 

(i) For every continuous driving function $\CX : \BR_+ \to \BR^d$ with $\CX(0) \in S$, the Skorohod problem in the orthant $S$ with reflection matrix $R$ has a unique solution. 

(ii) Take a vector $\mu \in \BR^d$ and a $d\times d$ positive definite symmetric matrix $A$. For every $x \in S$, there exists in the strong sense an $\SRBM^d(R, \mu, A)$ starting from $x$, and it is pathwise unique. These processes, starting from different $x \in S$, form  a Feller continuous strong Markov family. 
\end{prop}

Now we define a key concept: hitting non-smooth parts of the boundary $\pa S$ of the orthant $S$. (We already mentioned this in the Introduction.) This concept is a counterpart of triple and simultaneous collisions for systems of competing Brownian particles.

\begin{defn} The set 
$$
S^0 := \cup_{1 \le i < j \le d}(S_i\cap S_j) \subseteq \pa S
$$
is called  {\it non-smooth parts of the boundary $\pa S$}. An $S$-valued process $Z = (Z(t), t \ge 0)$ {\it hits non-smooth parts of the boundary at time $t$} if there exist $i, j = 1, \ldots, d$, $i \ne j$ such that $Z_i(t) = Z_j(t) = 0$. We say that the process $Z$ {\it hits non-smooth parts of the boundary} if there exists $t > 0$ such that it hits non-smooth parts of the boundary at time $t$. If such $t > 0$ does not exist, then we say that $Z$ {\it avoids non-smooth parts of the boundary}. 
\end{defn}

\subsection{Connection between an SRBM in the orthant and systems of competing Brownian particles}

In this subsection, we show that the gaps between adjacent particles in a system of competing Brownian particles form an SRBM in the orthant. 

\begin{defn} Consider a system of $N$ competing Brownian particles (a classical system from Definition~\ref{classical} or a system with asymmetric collisions from Definition~\ref{asymmdefn}). Let $Y_1, \ldots, Y_N$ be the ranked particles. Then the $\BR^{N-1}_+$-valued process 
$$
Z = (Z(t), t \ge 0),\ \ \ \ Z(t) = (Z_1(t), \ldots, Z_{N-1}(t))',
$$
defined by 
$$
Z_k(t) = Y_{k+1}(t) - Y_k(t),\ \ \ \ t \ge 0,\ \ \ \ k = 1,\ldots, N-1,
$$
is called the {\it gap process} for this system of competing Brownian particles. 
\label{gap}
\end{defn}

\begin{lemma} For a system of competing Brownian particles from Definition~\ref{asymmdefn}, the gap process is an $\SRBM^{N-1}(R, \mu, A)$, where
\begin{equation}
\label{R}
R = 
\begin{bmatrix}
1 & -q^-_2 & 0 & 0 & \ldots & 0 & 0\\
-q^+_2 & 1 & -q^-_3 & 0 & \ldots & 0 & 0\\
0 & -q^+_3 & 1 & -q^-_4 & \ldots & 0 & 0\\
\vdots & \vdots & \vdots & \vdots & \ddots & \vdots & \vdots\\
0 & 0 & 0 & 0 & \ldots & 1 & -q^-_{N-1}\\
0 & 0 & 0 & 0 & \ldots & -q^+_{N-1} & 1
\end{bmatrix},
\end{equation}
\begin{equation}
\label{A}
A = 
\begin{bmatrix}
\si_1^2 + \si_2^2 & -\si_2^2 & 0 & 0 & \ldots & 0 & 0\\
-\si_2^2 & \si_2^2 + \si_3^2 & -\si_3^2 & 0 & \ldots & 0 & 0\\
0 & -\si_3^2 & \si_3^2 + \si_4^2 & -\si_4^2 & \ldots & 0 & 0\\
\vdots & \vdots & \vdots & \vdots & \ddots & \vdots & \vdots\\
0 & 0 & 0 & 0 & \ldots & \si_{N-2}^2 + \si_{N-1}^2 & -\si_{N-1}^2\\
0 & 0 & 0 & 0 & \ldots & -\si_{N-1}^2 & \si_{N-1}^2 + \si_N^2
\end{bmatrix},
\end{equation}
\begin{equation}
\label{mu}
\mu = \left(g_2 - g_1, g_3 - g_4, \ldots, g_N - g_{N-1}\right)'.
\end{equation}
The matrix $R$ in~\eqref{R} is a nonsingular $\CM$-matrix. 
\end{lemma}

\begin{cor} For a classical system of competing Brownian particles from Definition~\ref{classical}, the gap process is an $\SRBM^{N-1}(R, \mu, A)$, where 
\begin{equation}
\label{R12}
R = 
\begin{bmatrix}
1 & -1/2 & 0 & 0 & \ldots\\
-1/2 & 1 & -1/2 & 0 & \ldots\\
0 & -1/2 & 1 & 0 & \ldots\\
\vdots & \vdots & \vdots & \vdots & \ddots\\
\end{bmatrix},
\end{equation}
while $A$ and $\mu$ are given by~\eqref{A} and~\eqref{mu}, respectively. 
\end{cor}

The proof can be found in \cite{BFK2005} for classical systems or in \cite{KPS2012} for systems with asymmetric collisions. However, for the sake of completeness we give the full proof here. 

\begin{proof} Using the equation~\eqref{SDEasymm}, we get the following equation for $Z_k = Y_{k+1} - Y_k$:
\begin{align*}
Z_k(t) =& Z_k(0) + \left(g_{k+1} - g_k\right)t + \si_{k+1}B_{k+1}(t) - \si_kB_k(t) \\ & + \left(q^+_{k+1} + q^-_k\right)L_{(k, k+1)}(t) - q^+_kL_{(k-1, k)}(t) - q^-_{k+1}L_{(k+1, k+2)}(t).
\end{align*}
Let 
$$
\ol{W}_k(t) = Z_k(0) + \left(g_{k+1} - g_k\right)t + \si_{k+1}B_{k+1}(t) - \si_kB_k(t),\ \ k = 1, \ldots, N-1,\ \ t \ge 0.
$$
Recall that $q^+_{k+1} + q^-_k = 1$ for $k = 1, \ldots, N-1$. Therefore,
$$
Z_k(t) = \ol{W}_k(t) + L_{(k, k+1)}(t) - q^+_kL_{(k-1, k)}(t) - q^-_{k+1}L_{(k+1, k+2)}(t).
$$
The $\BR^{N-1}$-valued process $\ol{W} = (\ol{W}_1, \ldots, \ol{W}_{N-1})'$ is an $(\CF_t)_{t \ge 0}$-Brownian motion in $N-1$ dimensions, with drift vector $\mu = (g_2 - g_1, \ldots, g_N - g_{N-1})'$ and covariance matrix $A$ given by~\eqref{A}. Indeed, $B_1, \ldots, B_N$ are i.i.d. standard Brownian motions. Therefore, 
$$
\langle \ol{W}_k\rangle_t = \langle \si_{k+1}B_{k+1}(t) - \si_kB_k(t)\rangle_t = \left(\si_k^2 + \si_{k+1}^2\right)t,
$$
$$
\langle \ol{W}_k, \ol{W}_{k+1}\rangle_t = \langle \si_{k+1}B_{k+1}(t) - \si_kB_k(t), \si_{k+2}B_{k+2}(t) - \si_{k+1}B_{k+1}(t)\rangle_t = -\si_{k+1}^2t,
$$
and $\langle \ol{W}_k, \ol{W}_l\rangle_t = 0$ for $|k - l| \ge 2$. 
The process $L_{(k, k+1)}$ for each $k = 1, \ldots, N-1$ satisfies the following conditions: it starts from zero, that is, $L_{(k, k+1)}(0) = 0$; it is nondecreasing, and can increase only when $Y_k = Y_{k+1}$, or, equivalently, when $Z_k = 0$. The rest is trivial. 
\end{proof}

\begin{rmk} A system of competing Brownian particles has a simultaneous collision at time $t$ if and only if the gap process hits non-smooth parts $S^0$ of the boundary $\pa S$ at time $t$. This is our method of proof: we state and prove results for an SRBM, and then we translate them into the language of systems of competing Brownian particles. 
\label{reduction}
\end{rmk}

\subsection{Main Results for an SRBM}

In this subsection, we state a necessary and sufficient condition for an SRBM a.s. to avoid non-smooth parts of the boundary. For the rest of this subsection, fix $d \ge 2$. Suppose $R$ is a $d\times d$ reflection nonsingular $\CM$-matrix. Fix a vector $\mu \in \BR^d$ and a $d\times d$ positive definite symmetric matrix $A$. Recall the notation $S = \BR^d_+$ and consider the process $Z = (Z(t), t \ge 0) = \SRBM^d(R, \mu, A)$, starting from some point $x \in S$. 

Let us give a necessary and sufficient condition for an SRBM a.s. not hitting non-smooth parts of the boundary $\pa S$ of the orthant $S$. 

\begin{thm}
\label{SRBMnonsmooth}
(i) Suppose the following condition holds: 
\begin{equation}
\label{SSineq}
r_{ij}a_{jj} + r_{ji}a_{ii} \ge 2a_{ij},\ \ 1 \le i, j \le d.
\end{equation}
Then with probability one, there does not exist $t > 0$ such that $Z$ hits non-smooth parts of the boundary at time $t$.

(ii) If the condition~\eqref{SSineq} is violated for some $1 \le i < j \le d$, then with positive probability there exists $t > 0$ such that $Z_i(t) = Z_j(t) = 0$. 
\end{thm}

\begin{rmk} The condition~\eqref{SSineq} can be written in the matrix form as $RD + DR^T \ge 2A$, where $D = \diag(A) = \diag(a_{11}, \ldots, a_{dd})$ is the diagonal $d\times d$-matrix with the same diagonal entries as $A$. The case when we have equality in~\eqref{SSineq} instead of inequality, is very important: the condition 
\begin{equation}
\label{SS}
RD + DR^T = 2A\ \ \ \Lra\ \ \ r_{ij}a_{jj} + r_{ji}a_{ii} = 2a_{ij},\ \ 1 \le i, j \le d,
\end{equation}
is called the {\it skew-symmetry condition}. This is a very important and well-studied case: see \cite{HW1987a}, \cite{HW1987b}, \cite{Wil1987}, \cite[Theorem 3.5]{Wil1995}. For example, under this condition, the SRBM has the product-of-exponentials stationary distribution. 
\end{rmk}

\begin{rmk} Whether an $\SRBM^d(R, \mu, A)$ a.s. avoids non-smooth parts of the boundary depends only on the matrices $R$ and $A$, not on the initial condition $Z(0)$ or the drift vector $\mu$. 
Some general results of this type are shown in subsection 3.2, Lemma~\ref{Independence}. But the actual probability of hitting non-smooth parts of the boundary, if it is positive, does depend on $\mu$ and the initial condition, see Remark~\ref{GIR}. 
\end{rmk}

\section{Proof of Theorem~\ref{SRBMnonsmooth}}

\subsection{Outline of the proof}


We can define a reflected Brownian motion not only in the orthant, but in more general domains: namely, in convex polyhedra, see \cite{DW1995}. Similarly to an SRBM in the orthant, this is a process which behaves as a Brownian motion in the interior of the domain and is reflected according to a certain vector at each face of the boundary. We can reduce an SRBM in the orthant with an arbitrary covariance matrix to a reflected Brownian motion in a convex polyhedron with identity covariance matrix. This construction is carried out in detail in subsection 3.5, Lemma~\ref{lemmatr}. 

Let us give a brief preview here. Consider an SRBM $Z = (Z(t), t \ge 0)$ in the orthant $\BR^d_+$ with covariance matrix $A$. 
Consider the process 
\begin{equation}
\label{lineartransf}
\ol{Z} = (\ol{Z}(t), t \ge 0),\ \ \ol{Z}(t) = A^{-1/2}Z(t),
\end{equation}
which is a reflected Brownian motion in the domain $A^{-1/2}\BR^d_+ := \{A^{-1/2}z\mid z \in \BR^d_+\}$ with identity covariance matrix. 

For a reflected Brownian motion in a polyhedral domain with identity covariance matrix, a sufficient condition (the {\it skew-symmetry condition}) for a.s. not hitting non-smooth parts of the boundary is known, see \cite[Theorem 1.1]{Wil1987}. Note that there are two forms of the skew-symmetry condition. One is for an SRBM in the orthant with arbitrary covariance matrix, which is~\eqref{SS}. The other is for a reflected Brownian motion in a convex polyhedron with identity covariance matrix, which was introduced in \cite{Wil1987}; in our paper, it is going to be given in~\eqref{SSgeneral}. In Lemma~\ref{equiv} we prove that under this linear transformation~\eqref{lineartransf}, these two conditions match. This justifies why they bear the same name. This allows us (in Lemma~\ref{first}) to prove part (i) of Theorem~\ref{SRBMnonsmooth} under the skew-symmetry condition~\eqref{SS}. 

Now, we need to show this for a more general condition~\eqref{SSineq}. We reduce this general case to the case of the skew-symmetry condition~\eqref{SS} by stochastic comparison (Lemma~\ref{compRR'}). We introduce an SRBM with new reflection matrix $\tilde{R}$ which satisfies the skew-symmetry condition and such that $\tilde{R} \ge R$. 

To prove part (ii), we first consider the case $d = 2$. The domain $A^{-1/2}\BR^2_+$ is in this case a {\it two-dimensional wedge}, which can be written in polar coordinates
$$
x_1 = r\cos\ta,\ \ x_2 = r\sin\ta,
$$
as 
$$
0 \le r < \infty,\ \ \xi_2 \le \ta \le \xi_1,
$$
where $\xi_1$, $\xi_2$ are angles such that $\xi_2 \le \xi_1 \le \xi_2 + \pi$. 
We mentioned that a reflected Brownian motion in this domain with zero drift vector and identity covariance matrix was studied in \cite{VW1985}, \cite{Wil1985a}, \cite{Wil1985b}, \cite{Wil1987b}. For this process, hitting non-smooth parts of the boundary means hitting the corner of the wedge (the origin). The result \cite[Theorem 2.2]{VW1985} gives a necessary and sufficient condition for a.s. avoiding the corner. Using the linear transformation~\eqref{lineartransf}, we can then translate these results for an SRBM in the positive quadrant with general covariance matrix. This proves (ii) for $d = 2$. 

To prove Theorem~\ref{SRBMnonsmooth} for the general $d$, we again use comparison techniques. We consider any two components $Z_i, Z_j$ of the process $Z = (Z(t), t \ge 0) = \SRBM^d(R, \mu, A)$, and compare them with a two-dimensional SRBM using Corollary~\ref{cutting}. 

Some parts of the calculations in this proof below have been done in certain previous articles. For example, the linear transformation $z \mapsto A^{-1/2}z$ and the way it transforms an SRBM in the orthant have been studied in the following articles: \cite[Section 9, Theorem 23]{HW1987b} (general dimension, under the skew-symmetry condition); \cite[Proposition 2]{KPS2012} (dimension $d = 2$). However, to make the exposition as lucid and self-contained as possible, we decided to do all calculations from scratch. 

\begin{rmk} In this artlce, we define a reflected Brownian motion in Definition~\ref{SRBMdefn} as a semimartingale. Similarly, in the article \cite{DW1995} a reflected Brownian motion in a convex polyhedron is defined in a semimartingale form; we present this in  Definition~\ref{semimartpolyhedron}. However, in the papers \cite{VW1985} and \cite{Wil1987}, a reflected Brownian motion is not given in a semimartingale form. Instead, it is defined as a solution to a certain submartingale problem: see Definition~\ref{defnsubmart}. We use the semimartingale definition, and in Lemma~\ref{submart} we prove that the semimartingale form of a reflected Brownian motion also satisfies the submartingale definition. This shows that we can indeed use the results from \cite{VW1985} and \cite{Wil1987}. 
\end{rmk}


\subsection{Girsanov removal of drift and independence of the initial conditions}

In this subsection, fix $d \ge 1$. Let $R$ be a $d\times d$ reflection nonsingular $\CM$-matrix. Let $A$ be a $d\times d$ symmetric positive definite matrix, and let $\mu \in \BR^d$. For every $x \in S$, denote by $\MP_x$ the probability measure corresponding to the $\SRBM^d(R, \mu, A)$ starting from $x$.

\begin{defn} For a nonempty subset $I \subseteq \{1, \ldots, d\}$, let 
$S_I = \{x \in S\mid x_i = 0,\ i \in I\}$. This is called an {\it edge} of the orthant $S$.
\end{defn}

For example, $S_{\{i, j\}} = S_i\cap S_j$ for $i \ne j$ is a piece of the non-smooth parts of the boundary $\pa S$. In this article, we are interested in an $\SRBM^d(R, \mu, A)$ hitting or avoiding these edges. But for this subsection, we shall work with a general edge $S_I$ of $S$. 

The main result of this subsection is that the property of an SRBM to a.s. avoid $S_I$ is independent of the starting point $x \in S$ and of the drift vector $\mu$. The proof is postponed until the end of this subsection. 

\begin{prop}  Let $Z = (Z(t)  \ge 0)$ be an $\SRBM^d(R, \mu, A)$. 
Let 
$$
p(x, R, \mu, A) = \MP_x\left(\exists\ t > 0:\ Z(t) \in S_I\right).
$$
Fix a $d\times d$ reflection nonsingular $\CM$-matrix $R$ and a positive definite symmetric $d\times d$ matrix $A$. Then one of these two statements is true:

\begin{itemize}
\item For all $\mu \in \BR^d$ and $x \in S$, we have: $p(x, R, \mu, A) = 0$: (the edge $S_I$ is avoided). 
\item For all $\mu \in \BR^d$ and $x \in S$, we have: $p(x, R, \mu, A) > 0$: (the edge $S_I$ is hit). 
\end{itemize}
\label{Independence}
\end{prop}

\begin{rmk}
\label{GIR}
We can reformulate Lemma~\ref{Independence} as follows: whether an $\SRBM^d(R, \mu, A)$ hits the edge $S_I$ does not depend on the initial conditions and the drift vector $\mu$; it depends only on the reflection matrix $R$ and the covariance matrix $A$. 

However, suppose $\SRBM^d(R, \mu, A)$ hits the edge $S_I$, so the probability $p(x, R, \mu, A)$ is positive. What is its exact value? This probability {\it does} depend on the drift vector $\mu$ and the initial condition $x \in S$. Let us give a one-dimensional example: a reflected Brownian motion on the positive half-line $\BR_+$ with no drift. With probability one, it hits the origin (which is the same as hitting the edge $S_{\{1\}}$). But a reflected Brownian motion on $\BR_+$ with positive drift $b$, starting from $x > 0$, hits the origin with probability $e^{-2bx}$, see \cite[Part 2, Section 2, formula 2.0.2]{BSBook}. This does depend on the drift $b$ and the initial condition $x$. 
\end{rmk}

\begin{defn} We say that an $\SRBM^d(R, \mu, A)$ {\it avoids non-smooth parts of the boundary} $\pa S$ of the orthant $S$ if it avoids every edge $S_I$ with $|I| = 2$. Otherwise, we say that an $\SRBM^d(R, \mu, A)$ {\it hits non-smooth parts of the boundary} $\pa S$. 
\end{defn}

From the discussion just above, we see: the property of hitting non-smooth parts of the boundary is independent of the initial condition $x$ and of the drift vector $\mu$. It depends only on $R$ and $A$. We can also see it from Theorem~\ref{SRBMnonsmooth}: the condition~\eqref{SSineq} involves only elements of $R$ and $A$. 

\subsection{Proof of Proposition~\ref{Independence}}

We split the proof of Lemma~\ref{Independence} in two steps. First, we show independence of a starting point $x \in S$ in Lemma~\ref{hit}, then of a drift vector $\mu \in \BR^d$ in Lemma~\ref{girsanov}, using the Girsanov transformation. 

\begin{lemma} For fixed parameters $R, \mu, A$ of an SRBM, we have:
either $p(x, R, \mu, A) = 0$ for all $x \in S$, or $p(x, R, \mu, A) > 0$ for all $x \in S$. In other words, either an $\SRBM^d(R, \mu, A)$ hits the edge $S_I$, or it avoids the edge $S_I$. 
\label{hit} 
\end{lemma}

\begin{proof}
Since the family of the processes $Z = (Z(t), t \ge 0) = \SRBM^d(R, \mu, A)$, starting from different points $x \in S$, is Feller continuous, the function
$$
f(z) := \MP_z\left(\exists t > 0: Z(t) \in S_I\right)
$$
is continuous on $S$. Let $P^t(x, C) = \MP_x(Z(t) \in C)$ be the transition function for the $\SRBM^d(R, \mu, A)$. By the Markov property, 
\begin{equation}
\label{mkv}
\MP_z\left(\exists t > 1: Z(t) \in S_I\right) = \int_SP^1(z, dy)f(y).
\end{equation}
But 
\begin{equation}
\label{mkv1}
\MP_z\left(\exists t > 1:\ Z(t) \in S_I\right) \le \MP_z\left(\exists t > 0:\ Z(t) \in S_I\right) = f(z).
\end{equation}
Combining~\eqref{mkv} and~\eqref{mkv1}, we have:
$$
\int_Sf(y)P^1(z, dy) \le f(z).
$$
Suppose for some $z_0 \in S$ we have: $f(z_0) > 0$. Since $f$ is continuous, there exists an open neighborhood $U$ of $z_0$ in $S$ such that $f(z) \ge f(z_0)/2 > 0$ for $z \in U$. But $U$ has positive Lebesgue measure, and so $P^1(z, U) > 0$ for $z \in S$. Therefore, $f(z) \ge P^1(z, U)f(z_0)/2 > 0$ for all $z \in S$. 

We have proved that if $f(z_0) > 0$ for at least one $z_0 \in S$, then $f(z) > 0$ for all $z \in S$. This completes the proof of the lemma. 
\end{proof}

\begin{lemma} Fix a nonempty subset $I \subseteq \{1, \ldots, d\}$. Then an $\SRBM^d(R, \mu, A)$ avoids $S_I$ if and only if an $\SRBM^d(R, 0, A)$ avoids $S_I$. 
\label{girsanov}
\end{lemma}

\begin{proof}
Using Lemma~\ref{hit}, without loss of generality, fix a starting point $z \in S$, the same for both processes. Let $Z = \SRBM^d(R, \mu, A)$, starting from $z$, and let $\ol{Z} = \SRBM^d(R, 0, A)$, starting from $z$. Let $P, \ol{P}$ be the distributions of the processes $Z, \ol{Z}$ on the space $C(\BR_+, \BR^d)$ of continuous functions $\BR_+ \to \BR^d$. For every $T > 0$, let $\CG_T$ be the $\sigma$-subalgebra of the Borel $\si$-algebra of $C(\BR_+, \BR^d)$, generated by the values of $x(s),\ 0 \le s \le T$ for all functions $x \in C(\BR_+, \BR^d)$. By the Girsanov theorem, for every $T > 0$, the restrictions $\left.P\right|_{\CG_T}$ and $\left.\ol{P}\right|_{\CG_T}$ are mutually absolutely continuous: they have common events of probability one. Therefore, the following statements are equivalent:
\begin{itemize}
\item With probability $1$, there is no $t \in (0, T]$ such that $Z(t) \in S_I$; \item With probability $1$, there is no $t \in (0, T]$ such that $\ol{Z}(t) \in S_I$. 
\end{itemize}
Suppose that with probability $1$, there is no $t > 0$ such that $Z_i(t) = 0$ for each $i \in I$; then for every $T > 0$, with probability $1$, there is no $t \in (0, T]$ such that $\ol{Z}_i(t) = 0$. Since $T > 0$ is arbitrary, we have: with probability $1$, there is no $t > 0$ such that $\ol{Z}_i(t) = 0$ for each $i \in I$. The converse statement is proved similarly. 
\end{proof}

\subsection{Stochastic comparison for an SRBM} 

Let us introduce the concept of {\it stochastic domination}, or {\it domination in law}. 

\begin{defn} Fix $d \ge 1$ and take two $\BR^d$-valued processes $Z = (Z(t), t \ge 0),\ \ol{Z} = (\ol{Z}(t), t \ge 0)$. We say that $Z$ is {\it stochastically dominated by} $\ol{Z}$ if for every $t \ge 0$ and $y \in \BR^d$ we have:
$$
\MP(Z(t) \ge y) \le \MP(\ol{Z}(t) \ge y).
$$
We say that $Z$ is {\it pathwise dominated by} $\ol{Z}$ if a.s. for all $t \ge 0$ we have: $Z(t) \le \ol{Z}(t)$. 
\end{defn}

If the processes $Z$ and $\ol{Z}$ are Markov, then by changing the probability space we can move from stochastic domination to pathwise domination, see \cite[Theorem 5]{Kamae1977}. There is a well-developed theory of stochastic domination and pathwise domination for processes with oblique reflection in the orthant. The most general result in this area is \cite[Theorem 4.1]{R2000}, see also \cite[Theorem 1.1(i)]{KR2012b}, \cite[Theorem 3.1]{Haddad2010}, \cite[Theorem 6(i)]{KW1996}. The following proposition is an immediate corollaries of \cite[Theorem 4.1]{R2000}. 

\begin{prop} Take two $d\times d$ reflection nonsingular $\CM$-matrices $R, \ol{R}$ such that $\ol{R} \le R$. Fix a vector $\mu \in \BR^d$ and a positive definite symmetric $d\times d$-matrix $A$. Let 
$$
Z = \SRBM^d(R, \mu, A),\ \ \ol{Z} = \SRBM^d(\ol{R}, \mu, A),\ \ \mbox{such that}\ \ Z(0) \preceq \ol{Z}(0).
$$
Then $Z \preceq \ol{Z}$. 
\label{compRR'}
\end{prop}

Here is another useful statement, proved in \cite[Corollary 3.6]{MyOwn2},  which is applied later in this article. 

\begin{cor} Let $d \ge 1$ and take a $d\times d$ reflection nonsingular $\CM$-matrix. Take a vector $\mu \in \BR^d$ and a positive definite symmetric $d\times d$-matrix $A$. Fix a nonempty subset $I \subseteq \{1, \ldots, d\}$ with $|I| = p,\ \ 1 \le p < d$. Let 
$$
Z = \SRBM^d(R, \mu, A),\ \ \ol{Z} = \SRBM^{p}([R]_I, [\mu]_I, [A]_I)
$$
such that $[Z(0)]_I = \ol{Z}(0)$ in law. Then $[Z]_I \preceq \ol{Z}$. 
\label{cutting}
\end{cor}

\subsection{An SRBM in a convex polyhedron}

Let us give a definition of an SRBM in convex polyhedra from \cite{DW1995}. Fix the dimension $d \ge 1$. First, let us define the state space, a polyhedral domain $\CP \subseteq \BR^d$. Fix $m \ge 1$, the number of edges. Let $n_1, \ldots, n_m \in \BR^d$ be unit vectors, and let $b_1, \ldots, b_m \in \BR$. The domain $\CP$ is defined by
\begin{equation}
\label{CPgen}
\CP := \{x \in \BR^d\mid n_i\cdot x \ge b_i,\ \ i  =1, \ldots, m\}.
\end{equation}
We assume that the interior of $\CP$ is nonempty and for each $j = 1, \ldots, m$ we have: 
\begin{equation}
\label{simplecondition}
\{x \in \BR^d\mid n_i\cdot x \ge b_i,\ \ i = 1, \ldots, m,\ i \ne j\} \ne \CP.
\end{equation}
In this case, the {\it edges} of $\CP$: 
$$
\CP_i = \{x \in \CP\mid n_i\cdot x = b_i\},\ \ i = 1, \ldots, m,
$$
are $(d-1)$-dimensional. Note that the vectors $n_i,\ i = 1, \ldots, m$, are inward unit normal vectors to each of the faces $\CP_1, \ldots, \CP_m$. Now, let us define an SRBM in the domain $\CP$. Fix the parameters of this SRBM: a vector $\mu \in \BR^d$, a $d\times d$ positive definite symmetric matrix $A$ and a $d\times m$-matrix $R$. 

\begin{defn} Fix a starting point $x \in \CP$. Take $B = (B(t), t \ge 0)$ to be a $d$-dimensional Brownian motion with drift vector $\mu$ and covariance matrix $A$, starting from $x$. Take an adapted continuous $\CP$-valued process $Z = (Z(t), t \ge 0)$
and an adapted continuous $\BR^m$-valued process
$$
L = (L(t), t \ge 0),\ \ \ L(t) = (L_1(t), \ldots, L_m(t))',
$$
such that:

(i) $Z(t) = B(t) + RL(t),\ \ t \ge 0$;

(ii) for every $i = 1, \ldots, m$, $L_i(0) = 0$, $L_i$ is nondecreasing and can increase only when $Z(t) \in \CP_i$. 

The process $Z$ is called a {\it semimartingale reflected Brownian motion} (SRBM) in the domain $\CP$ with {\it reflection matrix} $R$, {\it drift vector} $\mu$ and {\it covariance matrix} $A$. This process is denoted by  $\SRBM^d(\CP, R, \mu, A)$. 
\label{semimartpolyhedron}
\end{defn}

\begin{rmk} A particular case is an SRBM in the orthant $S$, which was introduced in Section 2: $\SRBM^d(R, \mu, A)$ is the same as $\SRBM^d(S, R, \mu, A)$.
\end{rmk}

Let $v_i$ be the $i$th column of $R$. An $\SRBM^d(\CP, R, \mu, A)$ behaves as a $d$-dimensional Brownian motion with drift vector $\mu$ and covariance matrix $A$ inside $\CP$. On each face $\CP_i$, it is reflected in the direction of the vector $v_i$. 

The paper \cite{DW1995} contains an existence and uniqueness result for an SRBM in $\CP$. We present this result in a slightly weaker version, which is still sufficient for this article.  For any nonempty subset $I \subseteq \{1, \ldots, m\}$, let $\CP_I := \cap_{i \in I}\CP_i$. A {\it positive linear combination} of vectors $u_1, \ldots, u_q$ is any vector $\al_1u_1 + \ldots + \al_qu_q$ with $\al_1, \ldots, \al_q > 0$. 

\begin{asmp} For every nonempty subset $I \subseteq \{1, \ldots, m\}$, we have: 

(i) $\CP_I \ne \varnothing$ and $\CP_{J} \subsetneq \CP_{I}$ for $I \subsetneq J \subseteq \{1, \ldots, m\}$;

(ii) there is a positive linear combination $v$ of vectors $v_i,\ i \in I$, such that $v\cdot n_i > 0,\ i \in I$;

(iii) there is a positve linear combination $n$ of vectors $n_i,\ i \in I$, such that $n\cdot v_i > 0,\ i \in I$. 
\label{Assumption}
\end{asmp}

The following result in an immediate corollary of \cite[Theorem 1.3]{DW1995}.

\begin{prop} Under Assumption~\ref{Assumption}, for every $x \in \CP$ there exists in the weak sense the process
$$
Z^{(x)} = (Z^{(x)}(t), t \ge 0) = \SRBM^d(\CP, R, \mu, A),
$$
starting from $Z^{(x)}(0) = x$, and it is unique in law. This family of processes $(Z^{(x)}, x \in \CP)$ is Feller continuous strong Markov. 
\label{existencepolyhedra}
\end{prop}

\begin{rmk} By Assumption~\ref{Assumption}(ii) applied to a subset $I = \{i\}$, we have: $v_i\cdot n_i > 0$. So we can normalize $v_i$ to make $v_i\cdot n_i = 1$. This is done by replacing $v_i$ by $k_iv_i$ for $k_i := (v_i\cdot n_i)^{-1}$ and replacing $L_i$ by $k_i^{-1}L_i$. 
Doing this for each $i = 1,\ldots, m$ is called {\it standard normalization}. The new reflection matrix is $\ol{R} = R\mathcal D$, where $\mathcal D = \diag( (v_1\cdot n_1)^{-1}, \ldots,  (v_m\cdot n_m)^{-1})$. If $\ol{v}_i = k_iv_i$ is the $i$th column of $\ol{R}$, we can decompose it into the sum
\begin{equation}
\label{decompos}
\ol{v}_i = n_i + q_i,
\end{equation}
where 
$$
q_i\cdot n_i = (\ol{v}_i - n_i)\cdot n_i = \ol{v}_i\cdot n_i - n_i\cdot n_i = 1 - 1 = 0,\ \ i = 1, \ldots, m.
$$
These vectors $n_i$ and $q_i$ are called the {\it normal and tangential components} of the reflection vector $\ol{v}_i$, respectively. 
\label{std}
Similar normalization was done for an SRBM in the orthant in \cite[Appendix B]{BDH2010}. 
\label{normalrmk}
\end{rmk}

As mentioned above, in the papers \cite{VW1985}, \cite{Wil1985a}, \cite{Wil1985b}, \cite{Wil1987}, \cite{Wil1987b}, reflected Brownian motion was defined as a solution to a certain submartingale problem. 
We are going to show that if an SRBM is defined in a semimartingale form, as in Definition~\ref{semimartpolyhedron}, then it is also a solution to this submartingale problem, so we can use the results of the papers mentioned
above. 

\begin{defn} Take a convex polyhedron $\CP$ from~\eqref{CPgen} and the parameters $R, \mu, A$ from Definition~\ref{semimartpolyhedron}. The symbol $C^2_c(\CP)$ stands for the family of twice continuously differentiable functions $f : \CP \to \BR$ with compact support. Define the following operator for functions $f \in C^2_c(\CP)$:
$$
\CL f := \frac12\SL_{i=1}^d\SL_{j=1}^da_{ij}\frac{\pa^2f}{\pa x_i\pa x_j} + \SL_{i=1}^d\mu_i\frac{\pa f}{\pa x_i}.
$$
A $\CP$-valued continuous adapted process $Z = (Z(t), t \ge 0)$ is called a {\it solution to the submartingale problem associated with} $(\CP, R, \mu, A)$, {\it starting from $x \in \CP$}, if:

(i) $Z(0) = x$ a.s.;

(ii) for every function $f \in C^2_c(\CP)$ which satisfies
$$
v_i\cdot \nabla f(x) \ge 0\ \ \mbox{for}\ \ x \in \CP_i, \ \ \mbox{for each}\ \ i = 1, \ldots, m,
$$
the following process is an $(\CF_t)_{t \ge 0}$-submartingale:
$$
\CM^f = (\CM^f(t), t \ge 0),\ \ \ \CM^f(t) = f(Z(t)) - \int_0^t\CL f(Z(s))\md s.
$$
\label{defnsubmart}
\end{defn}

\begin{lemma} The process $\SRBM^d(\CP, R, \mu, A)$, starting from $x \in \CP$, is a solution to the submartingale problem associated with $(\CP, R, \mu, A)$, starting from $x$. 
\label{submart}
\end{lemma}

The proof is postponed until the Appendix.

\subsection{Connection between an SRBM in the orthant and an SRBM in a polyhedron}

Using the linear transformation~\eqref{transform}, we can switch from an $\SRBM^d(R, \mu, A)$ in the orthant with covariance matrix $A$ to an $\SRBM^d$ in a convex polyhedron with identity covariance matrix. 

\begin{lemma} Consider the process $Z = (Z(t), t \ge 0)$, which is an 
$\SRBM^d(R, \mu, A)$. Define a new process $\ol{Z} = (\ol{Z}(t), t \ge 0)$ as follows: 
\begin{equation}
\label{transform}
\ol{Z}(t) = A^{-1/2}Z(t).
\end{equation}
\label{lemmatr}

(i) The process $\ol{Z}$ is an $\SRBM^d(\CP, \ol{R}, \ol{\mu}, I_d)$
in the convex polyhedron
\begin{equation}
\label{CP}
\CP := \{A^{-1/2}z\mid z \in S\} = \{\ol{z} \in \BR^d\mid A^{1/2}\ol{z} \ge 0\},
\end{equation}
with reflection matrix $\ol{R} := A^{-1/2}R$, drift vector $\ol{\mu} := A^{-1/2}\mu$ and covariance matrix $\ol{A} = I_d$. The domain $\CP$ is a convex polyhedron as in~\eqref{CPgen} with $m = d$ edges: $\CP_i := \{A^{-1/2}x\mid x \in S_i\},\ i = 1, \ldots, d$. This domain satisfies the condition~\eqref{simplecondition} and the Assumption~\ref{Assumption} (i). 

(ii) The standard normalization from Remark~\ref{std} gives us a new reflection matrix: $\tilde{R} := \ol{R}D^{1/2} = A^{-1/2}RD^{1/2}$. The $i$th column of $\tilde{R}$ is equal to
\begin{equation}
v_i := a_{ii}^{1/2}A^{-1/2}Re_i,\ \ i = 1, \ldots, d. 
\label{vi}
\end{equation}
The inward unit normal vector to the face $\CP_i$ is given by
\begin{equation}
\label{ni}
n_i = a_{ii}^{-1/2}A^{1/2}e_i,\ \ i = 1, \ldots, d.
\end{equation}
Furthermore, Assumption~\ref{Assumption}(ii) and (iii) is satisfied. 
\end{lemma}

\begin{proof} (i) We have: $Z(t) = B(t) + RL(t)$, where $B = (B(t), t \ge 0)$ is the driving Brownian motion for the process $Z$, and $L = (L(t), t \ge 0)$ is the vector of regulating processes. Here, $B$ is a $d$-dimensional Brownian motion with drift vector $\mu$ and covariance matrix $A$. Define $W = (W(t), t \ge 0)$ as $W(t) = A^{-1/2}B(t)$: this is a $d$-dimensional Brownian motion with drift vector $\ol{\mu} = A^{-1/2}\mu$ and identity covariance matrix. Then $\ol{Z}(t) := A^{-1/2}Z(t) = W(t) + A^{-1/2}RL(t)$. 
The state space of $\ol{Z}$ is the domain $\CP$, given in~\eqref{CP}. This is a convex polyhedron of the type~\eqref{CPgen}. Let us show 
it satisfies the condition~\eqref{simplecondition} and the Assumption~\ref{Assumption} (i). The linear transformation~\eqref{transform} is a bijection $\BR^d \to \BR^d$, hence it suffices to show that the orthant $S$ satisfies the condition~\eqref{simplecondition} and the Assumption~\ref{Assumption} (i), which is straightforward.

(ii) The face $\CP_i$ is spanned by vectors $A^{-1/2}e_j,\ j \in \{1, \ldots, d\}\setminus\{i\}$. The vector $n_i$ is normal to $\CP_i$, so we must have:  $n_i\cdot A^{-1/2}e_j = 0$. Since the matrix $A^{-1/2}$ is symmetric, $A^{-1/2}n_i\cdot e_j = 0$ for $ j \in \{1, \ldots, d\}\setminus\{i\}$. Therefore, $A^{-1/2}n_i = k_ie_i$ for some $k_i \in \BR$; so $n_i = k_iA^{1/2}e_i$. Let us find $k_i$ such that $n_i$ is inward oriented and has unit length. 

The inward orientation means that for any point $w$ in the relative interior of the face $\CP_i$, that is, in $\CP_i\setminus\left(\cup_{j\ne i}\CP_j\right)$, there exists $\eps > 0$ such that $w + \eps n_i \in \CP$. But the domain $\CP$ is obtained from the orthant $S = \BR^d_+$ by the linear transformation~\eqref{transform}. So we have: $w = A^{-1/2}z$ for some $z$ in the relative interior $S_i\setminus(\cup_{j\ne i}S_j)$ of the face $S_i$ of $\pa S$. We must have $w + \eps n_i \in \CP$. But 
$$
w + \eps n_i = A^{-1/2}\left(z + \eps k_iAe_i\right),\ \ \ \mbox{and}\ \ \  \CP = \{A^{-1/2}x\mid x \in S\}.
$$
Therefore, $w + \eps n_i \in \CP\ \Lra\ z + \eps k_iAe_i \in S$. Since $z \in S_i$, we have: $z_i = 0$, and $(Ae_i)_i = a_{ii} > 0$. But $z_i + \eps k_i(Ae_i)_i = (z + \eps k_iAe_i)_i \ge 0$, so we must have: $k_i \ge 0$. Now, let us find $|k_i|$ using the fact that $\norm{n_i} = 1$. Since the matrix $A^{1/2}$ is symmetric, we have: 
$$
\norm{A^{1/2}e_i} = \left[A^{1/2}e_i\cdot A^{1/2}e_i\right]^{1/2} = \left[A^{1/2}(A^{1/2}e_i)\cdot e_i\right]^{1/2} = \left[Ae_i\cdot e_i\right]^{1/2} = a_{ii}^{1/2}.
$$
But $\norm{n_i} = 1$, and $n_i = k_iA^{1/2}e_i$. So $|k_i|a_{ii}^{1/2} = 1$, and $|k_i| = a_{ii}^{-1/2}$. Earlier, we proved that $k_i \ge 0$. Therefore, $k_i = a_{ii}^{-1/2}$, which proves~\eqref{ni}. Now, let us show~\eqref{vi}. The $i$th column of $A^{-1/2}R$ is equal to $A^{-1/2}Re_i$. Using the fact that the matrix $A^{1/2}$ is symmetric, we have:  
\begin{align*}
A^{-1/2}Re_i\cdot n_i &= A^{-1/2}Re_i\cdot a_{ii}^{-1/2}A^{1/2}e_i = a_{ii}^{-1/2}A^{1/2}A^{-1/2}Re_i\cdot e_i \\ &= a_{ii}^{-1/2}Re_i\cdot e_i = a_{ii}^{-1/2}r_{ii} = a_{ii}^{-1/2}.
\end{align*}
Therefore, the standard normalization defined in Remark~\ref{normalrmk} leads to 
$$
v_i := a_{ii}^{1/2}A^{-1/2}Re_i,\ \ i = 1, \ldots, d,
$$
which proves~\eqref{vi}. 
Now, let us show that the Assumption~\ref{Assumption}(ii) and (iii) is satisfied. Note that the matrix $A^{1/2}$ is symmetric, so for every $i, j = 1, \ldots, d$ we have:
\begin{align*}
v_i\cdot n_j =& a_{ii}^{1/2}a_{jj}^{-1/2}A^{-1/2}Re_i\cdot A^{1/2}e_j = 
a_{ii}^{1/2}a_{jj}^{-1/2}A^{1/2}A^{-1/2}Re_i\cdot e_j \\ &= 
a_{ii}^{1/2}a_{jj}^{-1/2}Re_i\cdot e_j = a_{ii}^{1/2}a_{jj}^{-1/2}r_{ij}.
\end{align*}
Fix a nonempty subset $I \subseteq \{1, \ldots, d\}$ with $|I| = p$. Since the matrix $R$ is completely-$\CS$, the submatrix $[R]_I$ is an $\CS$-matrix. There exist positive numbers $\al_i$, $i \in I$, such that 
$\sum_{j\in I}r_{ij}\al_j > 0$ for $i \in I$. Take 
$n = \sum_{j \in I}a_{jj}^{1/2}\al_jn_j$. This is a positive linear combination of $n_j$, $j \in I$, and $v_i\cdot n = \sum_{j \in I}a_{ii}^{1/2}r_{ij}\al_j > 0$ for $i \in I$. This proves Assumption~\ref{Assumption}(iii). Similarly, the transposed matrix $R'$ is also completely-$\CS$ (this follows from Lemma~\ref{special}(ii)), so repeating this argument with $R'$ in place of $R$, we can prove Assumption~\ref{Assumption}(ii). 
\end{proof}

\subsection{Skew-symmetry condition}

Consider a reflected Brownian motion in a general convex polyhedron in general dimension $d \ge 2$. Then a sufficient condition for a.s. not hitting non-smooth parts of the boundary is given by \cite[Theorem 1.1]{Wil1987}. It is called the {\it skew-symmetry condition}. In the subsequent exposition, we define this condition in~\eqref{SSgeneral}, and show that it is equivalent (under the linear transformation~\eqref{transform}) to the skew-symmetry condition~\eqref{SS}. This is the reason why these two conditions have the same name. 

\begin{defn} Consider an $\SRBM^d(\CP, R, \mu, A)$ with $\mu = 0$ and $A = I_d$. Suppose the matrix $R$ is normalized, as described in Remark~\ref{normalrmk}. We say that the {\it skew-symmetry condition} holds if
\begin{equation}
\label{SSgeneral}
n_i\cdot q_j + n_j\cdot q_i = 0,\ \ 1 \le i,\, j \le m.
\end{equation}
\end{defn}

This justifies the name of this condition: the matrix $(n_i\cdot q_j)_{1 \le i, j \le m}$ must be skew-symmetric. 

We say that an SRBM $Z = (Z(t), t \ge 0)$ {\it hits non-smooth parts of the boundary} $\pa\CP$ at time $t > 0$ if there exist  $1 \le i < j \le m$ such that $Z(t) \in \CP_i\cap\CP_j$. This is a generalization of the concept of an SRBM in the orthant hitting non-smooth parts of the boundary. For an SRBM in a two-dimensional wedge, this is equivalent to hitting the corner of the wedge (the origin): a process $Z = (Z(t), t \ge 0)$ with values in this wedge {\it hits the corner} at time $t > 0$ if $Z(t) = 0$. 

\begin{prop} Under Assumption~\ref{Assumption} and the skew-symmetry condition~\eqref{SSgeneral}, an $\SRBM^d(\CP, R, \mu, A)$
starting from some point $x \in \CP\setminus\pa\CP$ in the interior of the polyhedral domain $\CP$ a.s. does not hit non-smooth parts of the boundary at any time $t > 0$. 
\label{generalnonsmooth}
\end{prop}

\begin{proof}
Follows from Lemma~\ref{submart}, Proposition~\ref{existencepolyhedra} 
and \cite[Theorem 1.1]{Wil1987}.
\end{proof}

The following lemma shows the equivalence of the two forms~\eqref{SS} and~\eqref{SSgeneral} of the skew-symmetry condition under the linear transformation~\eqref{transform}. 

\begin{lemma} Consider the process $Z = (Z(t), t \ge 0) = \SRBM^d(R, \mu, A)$. Let $\ol{Z}$ be the process defined by~\eqref{transform}. Then 
the skew-symmetry condition in the form~\eqref{SS} holds for $Z$ if and only if the skew-symmetry condition in the form~\eqref{SSgeneral} holds for $\ol{Z}$. 
\label{equiv}
\end{lemma}

\begin{proof} Suppose~\eqref{SS} is true. Using~\eqref{vi}, ~\eqref{ni} and the fact that $v_i = n_i + q_i,\ i = 1, \ldots, m$ (in this case $m = d$), we have:
\begin{align*}
n_i\cdot q_j &+ n_j\cdot q_i = n_i\cdot(v_j - n_j) + n_j\cdot (v_i - n_i) = n_i\cdot v_j - n_j \cdot v_i - 2n_i\cdot n_j \\ &= a_{ii}^{-1/2}A^{1/2}e_i\cdot a_{jj}^{1/2}A^{-1/2}Re_j + a_{jj}^{-1/2}A^{1/2}e_j\cdot a_{ii}^{1/2}A^{-1/2}Re_i - 2a_{ii}^{-1/2}a_{jj}^{-1/2}A^{1/2}e_i\cdot A^{1/2}e_j.
\end{align*}
Since the matrix $A^{1/2}$ is symmetric, we have:
\begin{align*}
a_{ii}^{-1/2}A^{1/2}e_i&\cdot a_{jj}^{1/2}A^{-1/2}Re_j = a_{ii}^{-1/2}a_{jj}^{1/2}\left(e_i\cdot A^{1/2}A^{-1/2}Re_j\right) \\ &= 
a_{ii}^{-1/2}a_{jj}^{1/2}\left(e_i\cdot Re_j\right) = a_{ii}^{-1/2}a_{jj}^{1/2}r_{ij},
\end{align*}
similarly
$$
a_{jj}^{-1/2}A^{1/2}e_j\cdot a_{ii}^{1/2}A^{-1/2}Re_i =  a_{jj}^{-1/2}a_{ii}^{1/2}r_{ji},
$$
and finally
\begin{align*}
a_{ii}^{-1/2}a_{jj}^{-1/2}A^{1/2}e_i&\cdot A^{1/2}e_j = a_{ii}^{-1/2}a_{jj}^{-1/2}\left(e_i\cdot A^{1/2}A^{1/2}e_j\right) \\ &= 
a_{ii}^{-1/2}a_{jj}^{-1/2}\left(e_i\cdot Ae_j\right) = a_{ii}^{-1/2}a_{jj}^{-1/2}a_{ij}.
\end{align*}
Therefore,
\begin{align*}
n_i\cdot q_j &+ n_j\cdot q_i = 
a_{ii}^{-1/2}a_{jj}^{1/2}r_{ij} + a_{jj}^{-1/2}a_{ii}^{1/2}r_{ji} - 2a_{ii}^{-1/2}a_{jj}^{-1/2}a_{ij} \\ &= a_{ii}^{-1/2}a_{jj}^{-1/2}\left[r_{ij}a_{jj} + r_{ji}a_{ii} - 2a_{ij}\right] = 0.
\end{align*}
The converse statement is proved similarly. 
\end{proof}

\subsection{An SRBM in a two-dimensional wedge}

A particular case of a polyhedral domain is a {\it two-dimensional wedge} (see Fig. 1), considered in \cite{VW1985}, \cite{Wil1985a}, \cite{Wil1985b}, \cite{Wil1987b}:
$$
\CV := \{(r\cos\ta, r\sin\ta)\mid 0 \le r < \infty,\ \xi_2 \le \ta \le \xi_1\}.
$$
Here, $\xi_2 < \xi_1 < \xi_2 + \pi$. Its {\it angle} is defined as $\xi := \xi_1 - \xi_2$. Its boundary $\pa\CV$ consists of two edges 
$$
\CV_i := \{(r\cos\xi_i, r\sin\xi_i)\mid 0 \le r < \infty\},\ \ i = 1, 2.
$$
 The edge  $\CV_1$ is called the {\it upper edge}, and the edge $\CV_2$ is called the {\it lower edge}. The difference between them is as follows: the shorter way to rotate $\CV_1$ to get $\CV_2$ is clockwise rather than counterclockwise. On each edge $\CV_i$, there is a reflection vector $v_i$, which forms the angle $\theta_i \in (-\pi/2, \pi/2)$ with the inward unit normal vector $n_i$. 

These angles are signed: positive angles $\theta_1$, $\theta_2$ are measured toward the vertex of $\CV$ (the origin). In other words, $\theta_1$ is the angle between $n_1$ and $v_1$, measured clockwise in the direction from $n_1$ to $v_1$. This means the following: if the shorter way to rotate the direction of $n_1$ to get the direction of $v_1$ is clockwise, then $\theta_1 > 0$; and if it is counterclockwise, then $\theta_1 < 0$. If $v_1$ and $n_1$ have the same direction, then $\theta_1 = 0$. Simlarly, $\theta_2$ is the angle between $n_2$ and $v_2$, measured counterclockwise from $n_2$ to $v_2$. 

We are interested in whether a reflected Brownian motion with zero drift vector and identity covariance matrix in this wedge hits the corner. A necessary and sufficient condition is established in \cite[Theorem 2.2]{VW1985}. 

\begin{figure}[t]
\centering
\begin{tikzpicture}
\draw  (10, 5)-- (0,0) -- (12, 0);
\draw [ ->] (6, 0)  -- (4.5, 1) node[anchor=east]{$v_2$};
\draw [ ->] (8, 4) -- (8, 1.5) node[anchor = south east]{$v_1$};
\draw [->] (6, 0) -- (6, 1.3) node[anchor = west]{$n_2$};
\draw [->] (8, 4) -- (8.7, 3) node[anchor = north]{$n_1$};
\draw  (1.5, 0) arc [start angle=0, end angle=26, radius=1.5]node[anchor = north west]{\ $\xi$};
\draw [ ->] (6, 0.5) arc [start angle=90, end angle=144, radius=0.5]node[anchor =  west]{\ \ \ $\theta_2$};
\draw [ <-] (8.02, 3.2) arc [start angle=270, end angle=305, radius=0.8]node[anchor = south]{\ \  $\theta_1$};
\draw (3, -1) node[anchor = west]{\textsc{Figure 1.} A two-dimensional wedge.};
\draw (1, -1.5) node[anchor = west]{Angles $\theta_1$ and $\theta_2$ are counted toward the vertex of the wedge};
\draw (1, -2) node[anchor = west]{Here, $n_1$ and $n_2$ are normal vectors, $v_1$ and $v_2$ are reflection vectors};
\draw (11, 0) node[anchor = south]{$\mathcal V_2$};
\draw (10, 5) node[anchor = north west]{$\mathcal V_1$};

\end{tikzpicture}
\end{figure}

\begin{prop} Consider an SRBM $Z = (Z(t), t \ge 0)$ in the wedge $\CV$ with $\mu = 0$ and $A = I_2$, starting from a point $x \in \CV\setminus\pa \CV$. 

(i) If $\ta_1 + \ta_2 > 0$, then a.s. there exists $t > 0$ such that $Z(t) = 0$.

(ii) If $\ta_1 + \ta_2 \le 0$, then a.s. there does not exist $t > 0$ such that $Z(t) = 0$. 
\label{wedge}
\end{prop}

\begin{proof} Follows from Lemma~\ref{submart}, Proposition~\ref{existencepolyhedra}, and Theorem 2.2 from \cite{VW1985}.
\end{proof}

In the case of two dimensions, $d = 2$, the linear transformation~\eqref{transform} leads to an SRBM in a two-dimensional wedge with identity covariance matrix. In the following lemma, we explicitly calculate the parameters of this SRBM: the angle $\xi$ of this wedge and the two angles $\theta_1$, $\theta_2$ of reflection. 

\begin{lemma} Suppose $Z = \SRBM^2(R, 0, A)$ and $\ol{Z}$ is the process defined by~\eqref{transform}. Then the polyhedral domain $\CP$ is in fact a wedge $\CV$ with the angle 
\begin{equation}
\label{xi}
\xi = \arccos\left[-\frac{a_{12}}{\sqrt{a_{11}a_{22}}}\right].
\end{equation}
The process $\ol{Z}$ is an SRBM in $\CV$ with zero drift vector, identity covariance matrix and the angles of reflection
\begin{equation}
\label{ta1}
\ta_1 = \arcsin\frac{a_{12} - a_{11}r_{21}}{\sqrt{a_{11}\left(a_{11}r_{21}^2 - 2a_{12}r_{21} + a_{22}\right)}},
\end{equation}
\begin{equation}
\label{ta2}
\ta_2 = \arcsin\frac{a_{12} - a_{22}r_{12}}{\sqrt{a_{22}\left(a_{22}r_{12}^2 - 2a_{12}r_{12} + a_{11}\right)}}.
\end{equation}
\end{lemma}

\begin{proof} First, note that $A^{-1/2}$ is a positive definite matrix, so it has a positive determinant. Therefore, the linear transformation~\eqref{transform} preserves the orientation of the plane $\BR^2_+$. The edges of this wedge are
$$
\mathcal V_i := A^{-1/2}S_i \equiv \{A^{-1/2}z \mid z \in S_i\},\ \ i = 1, 2. 
$$
In fact, $\mathcal V_1$ is the upper edge, and $\mathcal V_2$ is the lower edge. Indeed, for the original quadrant $S = \BR^2_+$, the edge $S_1 = \{x \in S\mid x_1 = 0\}$ is the upper edge, and the edge $S_2 = \{x \in S\mid x_2 = 0\}$ is the lower edge: in other words, the shorter way to rotate $S_1$ to get $S_2$ is clockwise rather than counterclockwise. But under the transformation~\ref{lineartransf}, $S_1$ is mapped to $\mathcal V_1$, and $S_2$ is mapped to $\CV_2$. This linear transformation preserves the orientation. Therefore, the shorter way to rotate $\CV_1$ to get $\CV_2$ is also clockwise rather than counterclockwise.  
The edge $\mathcal V_1$ has a directional vector $c_2 = A^{-1/2}e_2$, while the edge $\mathcal V_2$ has a directional vector $c_1 = A^{-1/2}e_1$. 
An important remark: consider the notation $\CP_i,\ i = 1, \ldots, d$, for edges of the polyhedron from Lemma~\ref{lemmatr}. Then our current notation $\CV_1$ and $\CV_2$ is consistent with this notation in the sense that 
\begin{equation}
\label{consistent}
\CV_1 = \CP_1\ \ \mbox{and}\ \ \CV_2 = \CP_2.
\end{equation}
The angle $\xi$ of the wedge is the angle between the edges $\mathcal V_1$ and $\mathcal V_2$. So $\xi$ is the angle between two vectors $c_1 = A^{-1/2}e_1$ and $c_2 = A^{-1/2}e_2$. Since the matrix $A^{-1/2}$ is symmetric, we have:
\begin{align*}
\cos\xi =& \frac{A^{-1/2}e_1\cdot A^{-1/2}e_2}{\norm{A^{-1/2}e_1}\norm{A^{-1/2}e_2}} = \frac{(A^{-1/2})^2e_1\cdot e_2}{\left[(A^{-1/2})^2e_1\cdot e_1\right]^{1/2}\left[(A^{-1/2})^2e_2\cdot e_2\right]^{1/2}} \\ &= \frac{A^{-1}e_1\cdot e_2}{\left[A^{-1}e_1\cdot e_1\right]^{1/2}\left[A^{-1}e_2\cdot e_2\right]^{1/2}}
= \frac{(A^{-1})_{12}}{(A^{-1})_{11}^{1/2}(A^{-1})_{22}^{1/2}}.
\end{align*}
But 
\begin{equation}
\label{inverse}
A^{-1} = \frac1{a_{11}a_{22} - a_{12}^2}
\begin{bmatrix}
a_{22} & -a_{12}\\
-a_{12} & a_{11}
\end{bmatrix}
\end{equation}
Therefore,
$$
\cos\xi = -\frac{a_{12}}{\sqrt{a_{11}a_{22}}},
$$
and we get~\eqref{xi}. Let us find the reflection angles $\ta_1$ and $\ta_2$. For the quadrant $S = \BR^2_+$, if we rotate the directional vector $e_2$ of the upper face $S_1$ clockwise by $\pi/2$, we get an inward normal vector to this face. But the linear transformation~\eqref{transform} preserves the orientation, so a similar statement is true for the wedge $\CV$: if we rotate the directional vector $c_2 = A^{-1/2}e_2$ of the upper face $\CV_1$ of the wedge clockwise by $\pi/2$, then we get an inward normal vector 
$$
\mathfrak n_1 \equiv \begin{bmatrix}(\mathfrak n_1)_1\\ (\mathfrak n_1)_2\end{bmatrix} := \begin{bmatrix}(c_2)_2\\ -(c_2)_1\end{bmatrix}
$$
Similarly, if we rotate the vector $c_1 = A^{-1/2}e_1$ by $\pi/2$ counterclockwise, we get an inward normal vector 
$$
\mathfrak n_2 \equiv \begin{bmatrix}(\mathfrak n_2)_1\\ (\mathfrak n_2)_2\end{bmatrix} := 
\begin{bmatrix}-(c_1)_2\\ (c_1)_1\end{bmatrix}
$$
to  $\mathcal V_1$. These are not unit vectors: $\mathfrak n_i \ne n_i$. In fact, $\norm{\mathfrak n_1} = \norm{c_2}$ and $\norm{\mathfrak n_2} = \norm{c_1}$. But $\mathfrak n_1$ has the same direction as $n_1$, and $\mathfrak n_2$ has the same direction as $n_2$. In other words, $\mathfrak n_1 = \norm{\mathfrak n_1}n_1$ and $\mathfrak n_2 = \norm{\mathfrak n_2}n_2$.

From Lemma~\ref{lemmatr} and~\eqref{consistent}, it follows that $v_1 = A^{-1/2}r_1$ and $v_2 = A^{-1/2}r_2$. These vectors are {\it not normalized} in the sense of Remark~\ref{normalrmk}. The angle $\ta_1$ between $n_1$ and $v_1$ has a sign: it is calculated toward the origin, or, in other words, counterclockwise from $n_1$ to $v_1$. But $n_1$ and $\mathfrak n_1$ have the same direction. Therefore, $\ta_1$ can be calculated as the signed angle from $\mathfrak n_1$ to $v_1$ in the counterclockwise direction:
\begin{align*}
\sin\ta_1 =& \frac{(\mathfrak n_1)_1(v_1)_2 - (\mathfrak n_1)_2(v_1)_1}{\norm{\mathfrak n_1}\norm{v_1}} = 
\frac{-(c_2)_2(v_1)_2  - (c_2)_1(v_1)_1}{\norm{c_2}\norm{v_1}} = 
-\frac{c_2\cdot v_1}{\norm{c_2}\norm{v_1}} \\ &= 
-\frac{A^{-1/2}e_2\cdot A^{-1/2}r_1}{\norm{A^{-1/2}e_2}\norm{A^{-1/2}r_1}} = 
-\frac{A^{-1/2}e_2\cdot A^{-1/2}r_1}{\left[A^{-1/2}e_2\cdot A^{-1/2}r_1\right]^{1/2}\left[A^{-1/2}e_2\cdot A^{-1/2}r_1\right]^{1/2}}
\end{align*}
Since the matrix $A^{-1/2}$ is symmetric, the last expression is equal to 
$$
-\frac{A^{-1}e_2\cdot r_1}{\left[A^{-1}e_2\cdot e_2\right]^{1/2}\left[A^{-1}r_1\cdot r_1\right]^{1/2}}.
$$
Using the formula~\eqref{inverse} for $A^{-1}$ and the fact that
$r_1 = (1, r_{21})'$, we have: 
$$
\sin\ta_1 = \frac{a_{12} - a_{11}r_{21}}{\sqrt{a_{11}\left(a_{11}r_{21}^2 - 2a_{12}r_{21} + a_{22}\right)}}.
$$
Similarly, we can calculate the angle $\ta_2$:
$$
\sin\ta_2 = \frac{a_{12} - a_{22}r_{12}}{\sqrt{a_{22}\left(a_{22}r_{12}^2 - 2a_{12}r_{12} + a_{11}\right)}}.
$$
Since $\ta_1, \ta_2 \in (-\pi/2, \pi/2)$, we get~\eqref{ta1} and ~\eqref{ta2}. 
\end{proof}

\subsection{Completion of the proof of Theorem~\ref{SRBMnonsmooth}}

By Lemma~\ref{hit}, without loss of generality we can assume an SRBM starts from some point $x \in S\setminus\pa S$, and $\mu = 0$. First, we prove (i) in the case of the skew-symmetry condition~\eqref{SS}, then move to the general case~\eqref{SSineq}. Then we prove (ii) in the case $d = 2$, and proceed to the case of the general dimension. 

\begin{lemma} Take an SRBM in the orthant $S$, starting from $x \in S\setminus\pa S$. Suppose it satisfies the skew-symmetry condition~\eqref{SS}. Then the statement of Theorem~\ref{SRBMnonsmooth}(i) is true.
\label{first}
\end{lemma}

\begin{proof} Apply the linear transformation~\eqref{transform} to $Z = (Z(t), t \ge 0) = \SRBM^d(R, 0, A)$. By Lemma~\ref{lemmatr}, we get an SRBM $\ol{Z} = (\ol{Z}(t), t \ge 0)$ in the polyhedron $\CS = A^{-1/2}S$, given by~\eqref{CP} with zero drift and identity covariance matrix. It was shown in Lemma~\ref{equiv} that the skew-symmetry condition~\eqref{SSgeneral} is true. Therefore, by Proposition~\ref{generalnonsmooth} the process $\ol{Z}$ a.s. does not hit non-smooth parts of the boundary $\pa\CS$ at any moment $t > 0$. Thus, the process $Z$ a.s. does not hit non-smooth parts of the boundary $\pa S$ at any moment $t > 0$. 
\end{proof}

\begin{lemma} Take an SRBM in the orthant $S$, starting from $x \in S\setminus\pa S$. Suppose it satisfies the condition~\eqref{SSineq}. Then the statement of Theorem~\ref{SRBMnonsmooth}(i) is true.
\end{lemma}

\begin{proof} Let us find another reflection nonsingular $\CM$-matrix $\tilde{R} = (\tilde{r}_{ij})_{1 \le i, j \le d}$ such that $R \ge \tilde{R}$, and the skew-symmetry condition~\eqref{SS} is true for an $\SRBM^d(\tilde{R}, 0, A)$. We need:
\begin{equation}
\label{SStilde}
\tilde{r}_{ij}a_{jj} + \tilde{r}_{ji}a_{ii} = 2a_{ij},\ i, j = 1, \ldots, d.
\end{equation}
Let $\tilde{r}_{ij} = 1$ for $i = j$. Then~\eqref{SStilde} is true for $i = j$. Let 
$$
\tilde{r}_{ij} = \frac1{a_{jj}}\left[2a_{ij} - r_{ji}a_{ii}\right],\ \tilde{r}_{ji} = r_{ji},\ \ 1 \le i < j \le d.
$$
This is well defined, since $a_{jj} > 0$ (because the matrix $A$ is positive definite). Also, $\tilde{r}_{ij} \le r_{ij}$, because $r_{ij}a_{jj} + r_{ji}a_{ii} \ge 2a_{ij}$. Since $\tilde{r}_{ij} \le r_{ij} \le 0$ for $i \ne j$, $\tilde{R}$ is a $\CZ$-matrix, so condition~\eqref{SStilde} holds. Therefore, by \cite[Theorem 2.5]{HornBook} (compare conditions 12 and 16), $\tilde{R}$ is a nonsingular $\CM$-matrix. Consider two processes $Z = \SRBM^d(R, \mu, A),\ \tilde{Z} = \SRBM^d(\tilde{R}, \mu, A)$, starting from the same initial condition $x \in S\setminus \pa S$. Then we have: $R$ and $\tilde{R}$ are $d\times d$ reflection nonsingular $\CM$-matrices, and $R \ge \tilde{R}$. By Proposition~\ref{compRR'}, we have: $\tilde{Z}$ is stochastically smaller than $Z$. By \cite[Theorem 5]{Kamae1977}, we can claim that a.s. for all $t > 0$ we have: $\tilde{Z}(t) \le Z(t)$ (possibly after changing the probability space). By Lemma~\ref{first}, the process $\tilde{Z}$ a.s. does not hit non-smooth parts of the boundary at any time $t > 0$. In other words, for every $1 \le i < j \le d$, we have: a.s. $\tilde{Z}_i(t) + \tilde{Z}_j(t) > 0$ for all $t > 0$. Therefore, a.s. $Z_i(t) + Z_j(t) > 0$ for all $t > 0$. Thus, with probability one the process $Z$ does not hit non-smooth parts of the boundary at any time $t > 0$. 
\end{proof}

Now, let us prove part (ii) of Theorem~\ref{SRBMnonsmooth}. We start with the case $d = 2$, then move to the general case. 

\begin{lemma}  Suppose we start an SRBM in two dimensions from a point $x \in S\setminus\pa S$ in the interior of $S$. Then the statement of Theorem~\ref{SRBMnonsmooth} (ii) is valid. 
\label{auxillary}
\end{lemma}

\begin{proof} Let $Z = (Z(t), t \ge 0) = \SRBM^2(R, 0, A)$. After the linear transformation~\eqref{transform}, we get the process $\ol{Z} = (\ol{Z}(t), t \ge 0)$ from~\eqref{transform}, which is an SRBM in a wedge. If we show that $\ta_1 + \ta_2 > 0$, then by Lemma~\ref{wedge} we have: a.s. there exists $t > 0$ such that $\ol{Z}(t) \equiv A^{-1/2}Z(t) = 0$; therefore, a.s. there exists $t > 0$ such that $Z(t) = 0$. But the angles $\ta_1$, $\ta_2$ are given in the equations~\eqref{ta1} and~\eqref{ta2}. Since $\ta_1, \ta_2 \in (-\pi/2, \pi/2)$, we have:
$$
\ta_1 + \ta_2 > 0\ \ \Lra\ \ \sin\ta_1 + \sin\ta_2 > 0,
$$
which can be written as
\begin{equation}
\label{7831}
\frac{a_{11}r_{21} - a_{12}}{\sqrt{a_{11}\left(a_{11}r_{21}^2 - 2a_{12}r_{21} + a_{22}\right)}} + \frac{a_{22}r_{12} - a_{12}}{\sqrt{a_{22}\left(a_{22}r_{12}^2 - 2a_{12}r_{12} + a_{11}\right)}} < 0.
\end{equation}
Then we have:
$$
r'_{12} := a_{11}^{-1/2}a_{22}^{1/2}r_{12},\ \ r'_{21} = a_{11}^{1/2}a_{22}^{-1/2}r_{21},\ \ \rho := a_{11}^{-1/2}a_{22}^{-1/2}a_{12}.
$$
We can rewrite the condition~\eqref{7831} as
$$
\frac{r'_{12} - \rho}{\sqrt{(r'_{12})^2 - 2\rho r'_{12} + 1}} + 
\frac{r'_{21} - \rho}{\sqrt{(r'_{21})^2 - 2\rho r'_{21} + 1}} < 0.
$$
Or, equivalently, $f(r'_{12} - \rho) + f(r'_{21} - \rho) < 0$, where
$$
f(x) := \frac{x}{\sqrt{x^2 + 1 - \rho^2}}.
$$
Note that the matrix $A$ is positive definite, so $\det A = a_{11}a_{22} - a_{12}^2  > 0$. Therefore, $\rho^2 < 1$. It is easy to show that the function $f$ is strictly increasing on $\BR$. In addition, this function is odd: $f(x) + f(-x) \equiv 0$. Therefore, $f(r'_{12} - \rho) + f(r'_{21} - \rho) < 0$ is equivalent to 
$$
(r'_{12} - \rho) + (r'_{21} - \rho) < 0\ \ \Lra\ \ r_{12}a_{22} + r_{21}a_{11} < 2a_{12}.
$$
\end{proof}

\begin{lemma} The statement (ii) of Theorem~\ref{SRBMnonsmooth} is valid in the case of general dimension, if we start an SRBM from a point $x \in S\setminus\pa S$ in the interior of $S$.
\end{lemma}

\begin{proof} Let $Z = \SRBM^d(R, 0, A)$. Assume now that the condition~\eqref{SSineq} is not true, and for some $1 \le i < j \le d$ we have: 
\begin{equation}
\label{violation}
r_{ij}a_{jj} + r_{ji}a_{ii} < 2a_{ij}.
\end{equation}
Consider the following two-dimensional SRBM: $\tilde{Z} = \SRBM^2([R]_I, 0, [A]_I)$, where $I = \{i, j\}$. Applying Corollary~\ref{cutting} to $I := \{i, j\}$, we get:  $[Z]_I \preceq \tilde{Z}$. By \cite[Theorem 5]{Kamae1977}, we can switch from stochastic comparison to pathwise comparison: after changing the probability space, we can claim that a.s. for all $t > 0$ we have: $[Z(t)]_I \le \tilde{Z}(t)$. By  Lemma~\ref{auxillary}, with positive probability, there exists $t > 0$ such that $\tilde{Z}_i(t) = \tilde{Z}_j(t) = 0$. Therefore, with positive probability there exists $t > 0$ such that $Z_i(t) = Z_j(t) = 0$. 
\end{proof}

\section{Proof of Theorems~\ref{classicalthm} and~\ref{thmasymm}}

Theorem~\ref{thmasymm} can be easily deduced from  Theorem~\ref{SRBMnonsmooth}. First, let us prove part (i) of Theorem~\ref{thmasymm}. We need to rewrite the condition~\eqref{SSineq} for concrete matrices $R$ and $A$ arising from competing Brownian particles, given by~\eqref{R} and~\eqref{A}. Take $i, j = 1, \ldots, N-1$ and consider the condition
\begin{equation}
\label{m}
r_{ij}a_{jj} + r_{ji}a_{ii} \ge 2a_{ij}.
\end{equation}
If $i = j$, then~\eqref{m} is always true, because for such $i$, $j$ we have: $r_{ij} = r_{ji} = 1$, and $a_{ii} = a_{ij} = a_{jj} = \si_i^2 + \si_{i+1}^2$. If $|i - j| \ge 2$, then~\eqref{m} is also always true, since $r_{ij} = r_{ji} = a_{ij} = 0$. 
Since the left-hand side and the right-hand side of~\eqref{m} remain the same if we swap $i$ and $j$, we need only to check this condition for $j = k$, $i = k-1$, where $k = 2, \ldots, N-1$. We get:
$$
r_{ij} = -q^-_{k},\ \ r_{ji} = -q^+_{k},\ a_{jj} = \si_{k}^2 + \si_{k+1}^2,\ a_{ii} = \si_{k-1}^2 + \si_{k}^2,\ a_{ij} = -\si_{k}^2.
$$
Therefore, the condition~\eqref{m} takes the form
$$
-q^-_{k}\left(\si_{k}^2 + \si_{k+1}^2\right) - q^+_{k}\left(\si_{k-1}^2 + \si_{k}^2\right) \ge -2\si_{k}^2.
$$
This is equivalent to
\begin{equation}
\label{nextn}
\left(2 - q^-_{k} - q^+_{k}\right)\si_{k}^2 \ge q^-_{k}\si_{k+1}^2 + q^+_{k}\si_{k-1}^2.
\end{equation}
Note that $q^-_{k} + q^+_{k+1} = 1$ and $q^+_{k} + q^-_{k-1} = 1$. Therefore, we can rewrite~\eqref{nextn} as in~\eqref{conditionasymm}. This proves part (i) of Theorem~\ref{thmasymm}. Now, let us prove part (ii) of this theorem. Since the condition~\eqref{SSineq} is automatically valid for $i = j$ and for $|i-j| \ge 2$, it can be violated only for 
$i = j - 1$. Suppose it does not hold for $j = k$ and $i = k - 1$, where $k = 2, \ldots, N-1$ is some index. Then with positive probability, there exists $t > 0$ such that
$$
Z_{k-1}(t) = Z_k(t) = 0,
$$
which can be written as
$$
Y_{k-1}(t) = Y_k(t) = Y_{k+1}(t).
$$
This means that with positive probability, there is a triple collision between particles with ranks $k-1$, $k$ and $k+1$. This completes the proof of Theorem~\ref{thmasymm}. 

Theorem~\ref{classicalthm} is simply a corollary of Theorem~\ref{thmasymm}: just plug parameters of collision $q^{\pm}_k = 1/2$, $k = 1, \ldots, N$ into the inequality~\eqref{conditionasymm}.

\begin{rmk} Let us explain the meaning of Corollary~\ref{interesting} informally. Consider the gap process of a system of competing Brownian particles from Definition~\ref{asymmdefn}. This is an SRBM $Z = (Z(t), t \ge 0)$ in the orthant with reflection matrix $R$ and covariance matrix $A$, given by~\eqref{R} and~\eqref{A}. In this case, the condition~\eqref{SSineq} can be violated only for $i = j - 1$, because for $i = j$ and $|i - j| \ge 2$ it is automatically true. 

When $Z_i(t) = Z_j(t) = 0$ for $1 \le i < j \le d$, this corresponds to a simultaneous collision at time $t$ in this system of competing Brownian particles: $Y_{i}(t) = Y_{i+1}(t)$ and $Y_{j}(t) = Y_{j+1}(t)$. But if, in addition, we know that $i = j - 1$, then this is a particular case of a simultaneous collision: namely, a {\it triple collision} between particles with ranks $j-1$, $j$ and $j+1$. This implies that if the condition~\eqref{SSineq} does not hold, then with positive probability there occurs a simultaneous collision of a special kind: a triple collision. This is the reason why Corollary~\ref{interesting} is true. 

\end{rmk}


\section{Appendix}

\subsection{Proof of Lemma~\ref{special}}
 (i) $\Ra$ (iii). Use \cite[Theorem 2.5.3]{HornBook}. Since $R$ is completely-$\CS$, it satisfies condition 12 from this theorem. Therefore, it satisfies condition 2 from this theorem. We get the following representation: $R = \ga I_d - Q$, where $\ga := \max_{1 \le i \le d}r_{ii} = 1$, and a $d\times d$-matrix $Q$ is nonnegative with spectral radius less than one. (See the beginning of \cite[Section 2.5.4]{HornBook}.) 

(iii) $\Ra$ (ii). By \cite[Section 7.10]{MeyerBook}, we can represent $R^{-1}$ as Neumann series: 
$$
R^{-1} = I_d + Q + Q^2 + \ldots
$$
Since $Q$ is nonnegative, $R^{-1}$ is also nonnegative, and the diagonal elements of $R^{-1}$ are strictly positive (and even greater than or equal to $1$). 

(ii) $\Ra$ (i). Apply \cite[Theorem 2.5.3]{HornBook} again: condition 17 implies condition 12. Therefore, there exists $x \in \BR^d$, $x > 0$ such that $Rx > 0$, so $R$ is an $\CS$-matrix. Take a principal submatrix $\tilde{R}$ of $R$ and show that it is also an $\CS$-matrix. Let $\tilde{R} := [R]_I$, where $I \subsetneq \{1, \ldots, d\}$ is a nonempty set. Let $\tilde{x} := [x]_I$. Then $r_{ij} \le 0$ for $i \in I$ and $j \in I^c := \{1, \ldots, d\}\setminus I$, and 
$$
\bigl(\tilde{R}\tilde{x}\bigr)_i = \SL_{j\in I}r_{ij}x_j \ge \SL_{i=1}^dr_{ij}x_j = (Rx)_i > 0,\ \ i \in I.
$$
Therefore, $\tilde{x} > 0$ and $\tilde{R}\tilde{x} > 0$. So every principal submatrix of $R$ is an $\CS$-matrix, which proves that the matrix $R$ is completely-$\CS$.

\subsection{Proof of Lemma~\ref{submart}} 
Recall that the process $Z = (Z(t), t \ge 0)$ which is an $\SRBM^d(\CP, R, \mu, A)$ can be represented as $Z(s) = B(t) + RL(t)$. Here, $B = (B(t), t \ge 0)$ is a $d$-dimensional Brownian motion with drift vector $\mu$ and covariance matrix $A = (a_{ij})_{1 \le i, j \le d}$; $R = (r_{ij})$ is an $m\times d$-matrix, and $L = (L_1,\ldots, L_m)'$, where each $L_i$ is nondecreasing. Therefore, the mutual variation of the components of $Z$ is calculated as follows: $\langle Z_i, Z_j\rangle_t = a_{ij}t$, for $i, j = 1, \ldots, d$. The process $(B_i(s) - \mu_is, s \ge 0)$ is a one-dimensional driftless Brownian motion.
Since $f \in C^2_c(\CP)$, the following process is a martingale:
$$
M(t) = \SL_{i=1}^d\int_0^t\frac{\pa f}{\pa x_i}(Z(s))\md (B_i(s) - \mu_is).
$$
Apply the It\^o-Tanaka formula to $f(Z(t))$: 
\begin{align*}
f(Z(t)) - f(Z(0)) = &\SL_{i=1}^d\int_0^t\frac{\pa f}{\pa x_i}(Z(s))\md Z(s) + \frac12\SL_{i=1}^d\SL_{j=1}^d\int_0^t\frac{\pa^2f}{\pa x_i\pa x_j}(Z(s))\md\langle Z_i, Z_j\rangle_s \\ & = 
\SL_{i=1}^d\int_0^t\frac{\pa f}{\pa x_i}(Z(s))\md\left(B_i(s) - \mu_is\right) + \SL_{i=1}^d\int_0^t\frac{\pa f}{\pa x_i}(Z(s))\mu_i\md s
\\ & + \frac12\SL_{i=1}^d\SL_{j=1}^da_{ij}\int_0^t\frac{\pa^2f}{\pa x_i\pa x_j}(Z(s))\md s + \SL_{i=1}^d\int_0^t\frac{\pa f}{\pa x_i}(Z(s))\md\left[\SL_{j=1}^mr_{ij}L_j(s)\right] \\ & = M(t) + \int_0^t\CL f(Z(s))\md s + \SL_{i=1}^d\SL_{j=1}^m\int_0^tr_{ij}\frac{\pa f}{\pa x_i}(Z(s))\md L_j(s) \\ & = M(t) + \int_0^t\CL f(Z(s))\md s + \SL_{j=1}^m\int_0^tv_j\cdot \nabla f(Z(s))\md L_j(s).
\end{align*}
The third term in the last sum is nondecreasing. Indeed, for each $j = 1, \ldots, m$, the process $L_j$ is nondecreasing, and it can increase only when $Z(s) \in \CP_j$. But in this case, $v_j\cdot\nabla f(Z(s)) \ge 0$. The rest is trivial.

\section*{Acknoweldgements}

I would like to thank \textsc{Ioannis Karatzas}, \textsc{Soumik Pal} and \textsc{Ruth Williams}, as well as an anonymous referee, for help and useful discussion. This research was partially supported by  NSF grants DMS 1007563, DMS 1308340, and DMS 1405210.


\bibliographystyle{plain}

\bibliography{aggregated}

\end{document}